\documentclass[final,5p,times]{elsarticle}
\usepackage{multirow}
\usepackage{booktabs}
\usepackage{amssymb}
\usepackage{mathrsfs}        
\usepackage{amsthm}          
\usepackage{amsmath}         
\usepackage{epstopdf}
\biboptions{numbers,sort&compress}
\usepackage[colorlinks, citecolor=blue]{hyperref}
\hypersetup{colorlinks=false}                       
\usepackage{xcolor}  
\definecolor{b}{rgb}{0,0,1}
\usepackage{caption2}                   
\theoremstyle{plain}
\newtheorem{theorem}{Theorem}

\newtheorem{corollary}{Corollary}
\newtheorem{lemma}{Lemma}

\theoremstyle{definition}

\newtheorem{remark}{Remark}

\journal{IFAC Journal of Systems and Control}

\begin{document}

\begin{frontmatter}



\title{Adaptive backstepping control for FOS with nonsmooth nonlinearities}
\phantomsection
\addcontentsline{toc}{title}{Adaptive backstepping control for FOS with nonsmooth nonlinearities}


\author{Dian Sheng}
\author{Yiheng Wei}
\author{Songsong Cheng}
\author{Yong Wang\corref{cor1}}
\ead{yongwang@ustc.edu.cn}
\cortext[cor1]{Yong Wang is the corresponding author. }
\address{Department of Automation, University of Science and Technology of China, Hefei, 230026, China}

\begin{abstract}
This paper proposes an original solution to input saturation and dead zone of fractional order system. To overcome these nonsmooth nonlinearities, the control input is decomposed into two independent parts by introducing an intermediate variable, and thus the problem of dead zone and saturation transforms into the problem of disturbance and saturation afterwards. With the procedure of fractional order adaptive backstepping controller design, the bound of disturbance is estimated, and saturation is compensated by the virtual signal of an auxiliary system as well. In spite of the existence of nonsmooth nonlinearities, the output is guaranteed to track the reference signal asymptotically on the basis of our proposed method. Some simulation studies are carried out in order to demonstrate the effectiveness of method at last.
\end{abstract}


\end{frontmatter}


\section{Introduction}\label{Section 1}
When we deal with practical control problems, it is inevitable to face the troubles caused by the presence of real physical components, which often contains nonsmooth nonlinearities \cite{Zhou:2008Book}. Nonsmooth nonlinearities always lead to such unexpected result that they have become an independent and new field of research. Of all kinds of nonsmooth nonlinearities, the saturation and dead zone are the most common and significant since almost all actuators are enslaved to such two nonlinearities to some extent. As a static input-output characteristic, dead zone often appears in gears \cite{Zuo:2015ISAT,Zuo:2016TCST}, valves \cite{Xu:2016ISAT}, motors \cite{Liu:2013NODY}, robot arms \cite{Jiang:2015NODY} and even elevators of aircraft \cite{Xu:2015NODY}. As for saturation, the motor speed is limited due to physical constraints, the output of the operational amplifier won't be larger than its supply voltage and the finite word length results in overflow in computer. With the purpose of compensating dead zone and saturation, many scholars are dedicated to these nonlinearities and obtain abundant results \cite{Carnevale:2014TAC,Shen:2016TIE,Chen:2012AMM,Chen:2012JVC}. Be that as it may, when it comes to fractional order systems (FOS), there are few results about how to deal with dead zone and saturation.

As is known to all, fractional order system and control develop rapidly during last decades. There are two factors that lead to this prosperity. First, fractional order calculus could describe many engineering plants and processes precisely. The second factor is the superiority of fractional order controllers, i.e., design freedom and robust ability \cite{Chen:2016CNSNS,Du:2016NODY,Yin:2014Automatica,Chen:2016NODY,Li:2009Automatica}. Therefore, a large number of exciting results about FOS such as stability analysis \cite{Lu:2010TAC,Li:2012JFI,Shen:2015TAC}, controllability and observability \cite{Yan:2011JFI}, signal processing \cite{Acharya:2014SP}, numerical computation \cite{Wei:2016ISAT,Wei:2014IJCAS,Sheng:2011JFI}, system identification \cite{Petras:2012TSP,Hu:2016IJSS}  and   controller synthesis \cite{Luo:2012Automatica} have been achieved. As the FOS develop, some potential problems, especially dead zone and saturation, which were once ignored, gradually attract the attention of scholars now.

Although great efforts have been made to control FOS with dead zone and saturation, there are still few valuable results so far. Machodo reveals the superior performance of fractional order controller in the control of systems with nonlinear phenomena and proves it separately in the presence of saturation, dead zone, hysteresis and relay \cite{Machado:1997SAMS}. By exploiting the saturation function with sector bounded condition, Lim puts forward a method that makes fractional order linear systems with input saturation asymptotically stable \cite{Lim:2013TAC}. Based on $t^{-\alpha}$ asymptotical stability, Shahri analyzes the stability of the same system by means of direct Lyapunov method \cite{Shahri:2015AML}. In recent paper \cite{Shahri:2015IJAC}, Shahri also studies the similar problem but disturbance rejection is taken into consideration this time. Considering both their works on linear system, Luo lately discusses the saturation problem of nonlinear FOS \cite{Luo:2014MPE}. In the case of FOS with dead zone, sliding mode control and switching adaptive control have been applied by confining the control input to a sector area \cite{Abooee:2013CTA,Tian:2014Entropy,Roohi:2015Complexity}. The indirect Lyapunov method is introduced to control fractional order micro electro mechanical resonator with dead-zone input \cite{Tian:2015JCND}. Even if some significant works have been done, it's still far from perfection. There remains room to improve current research.
\vspace{-0.1cm}
\begin{itemize}
\item Dead zone and saturation are studied separately while they often take effect at the same time.
\item The form of dead zone and saturation should not be confined.
\item The order of the FOS could be extended to incommensurate one.
\item In addition to stabilization, the goal of tracking is supposed to be achieved.
\item The uncertain nonlinear FOS  deserves deeply researching.
\item External disturbance cannot be ignored in reality.
\end{itemize}

Therefore, this paper proposes a control method for uncertain nonlinear incommensurate FOS in the presence of dead zone and saturation. To the best of our knowledge, no scholar has ever investigated the input saturation and dead zone of FOS simultaneously. Our proposed method is creative and original not only in FOS but also integer order systems. First, by introducing intermediate variable, the input saturation and dead zone are decomposed into two independent parts where the dead zone part needs decomposing further. Thus the problem of input saturation and dead zone is transformed into the the problem of input saturation and bounded disturbance. Second, a fractional order auxiliary system is constructed to generate virtual signals which are used for compensating saturation. Third, based on fractional order adaptive backstepping control (FOABC) \cite{Wei:2015NEU,Wei:2016NODY,Ding:2014CTA,Ding:2015NODY}, the objective of tracking reference signal has been achieved in the end. Considering not all parameters of input saturation and dead zone available in practice, the entirely unknown case of nonsmooth nonlinearities is seriously studied as well.

The remainder of this article is organized as follows. Section \ref{Section 2} introduces the problem formulation and provides some basic knowledge for subsequent use. After model transformation and input decomposition, adaptive backstepping control strategy is recommended for the known and unknown cases of parameters in Section \ref{Section 3}. In Section \ref{Section 4}, simulation results are provided to illustrate the validity of the proposed approach. Conclusions are given in Section \ref{Section 5}.

\section {Preliminaries}\label{Section 2}
Let us consider the following parameter strict-feedback of uncertain nonlinear incommensurate FOS with nonsmooth nonlinearities and disturbance
\begin{eqnarray}\label{Eq2-1}
\left\{ \begin{array}{rl}
{{\mathscr D}^{{\alpha _i}}}{x_i} =& \hspace{-8pt} {d_i}{x_{i + 1}} + {\psi _i}({x_1}, x_2 , \cdots , {x_i}) \\
& \hspace{-8pt} + \varphi _i^{\rm{T}}({x_1}, x_2, \cdots , {x_i})\theta , ~i = 1, \cdots ,n - 1,\\
{{\mathscr D}^{{\alpha _n}}}{x_n} =& \hspace{-8pt} bu(v) + {\psi _n}(x) + \varphi _n^{\rm{T}}(x)\theta  + d,\\
y =& \hspace{-8pt} {x_1},
\end{array} \right.
\end{eqnarray}
where $0<\alpha_i<1 ~(i=1, \cdots ,n)$ are the system incommensurate fractional order, $\theta  \in {{\mathbb R}^q}$ is a constant and unknown vector, $b$ is the input coefficient, $d_i ~(i=1, \cdots ,n-1)$ are known nonzero real constants, ${\psi _i}\left(  \cdot  \right) \in {\mathbb R}$ and ${\varphi _i}\left(  \cdot  \right) \in {{\mathbb R}^q}$ ~($i = 1, \cdots ,n)$ are known nonlinear functions, $x = {\left[ {{x_1} ~\cdots~ {x_n}} \right]^{\rm{T}}} \in {{\mathbb R}^n}$ is the pseudo system state vector, and $d$ is the time-varying disturbance within unknown bound $|d| < D$. $u(v) \in {\mathbb R}$ represents the actual control input restrained by dead zone as well as saturation, and $v \in {\mathbb R}$ is the desired control input, namely,
\begin{eqnarray}\label{Eq2-2}
u = \left\{ \begin{array}{cl}
{U_{up}} &, ~v(t) \ge {w_{M}},\\
m[v(t) - {b_r}] &, ~{b_r} \le v(t) < {w_{M}},\\
0 &, ~{b_l} < v(t) < {b_r},\\
m[v(t) - {b_l}] &, ~{w_{m}} < v(t) \le {b_l},\\
{U_{low}} &, ~v(t) \le {w_{m}},
\end{array} \right.
\end{eqnarray}
with ${w_{M}} = \frac{{{U_{up}}}}{m} + {b_r}$ and ${w_{m}} = \frac{{{U_{low}}}}{m} + {b_l}$. The slope $m$ is possibly known, but the parameters of such nonsmooth nonlinear input, $U_{up}, ~U_{low}, ~b_l, ~b_r$, are entirely unknown. If slope $m$ is unknown, all parameters of dead zone and saturation are unknown as well, which leads to a big trouble that we will study later. To show the relationship between $u$ and $v$ visually, the phase diagram is drawn in Fig \ref{Fig 1}.
\begin{figure}[htbp]
  \centering
  \includegraphics[width=0.25\textwidth]{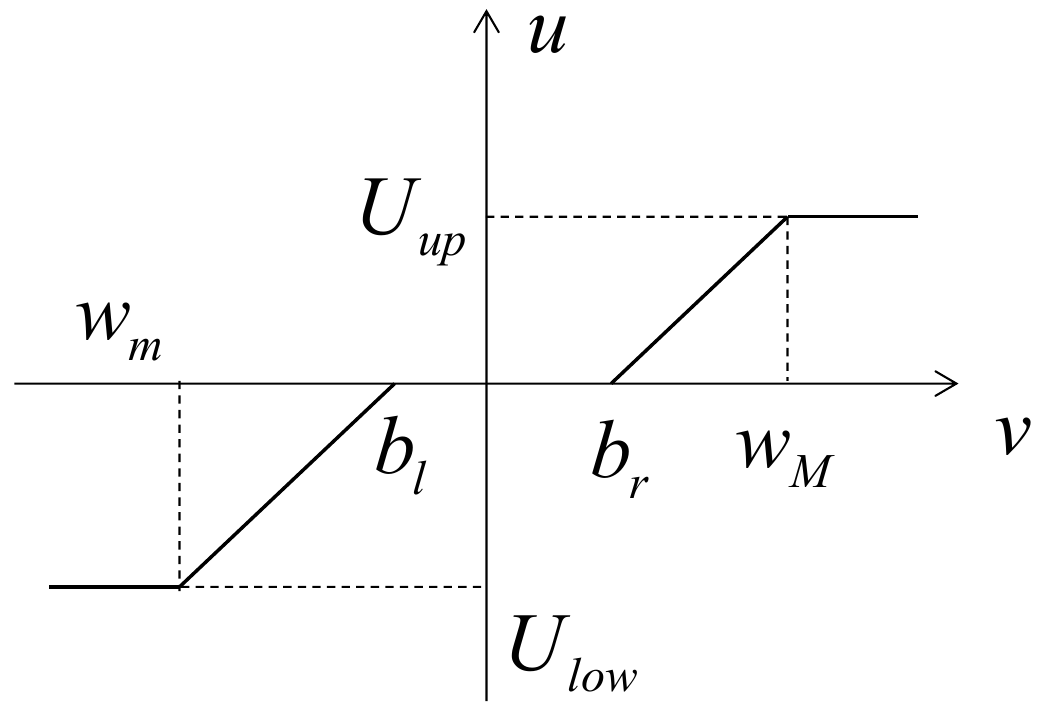}
  \caption{Control input subject to dead zone and saturation.}\label{Fig 1}
\end{figure}

The objective of this study is to develop an effective approach to enable the output $y$ to track the reference signal $r$ asymptotically under the following assumptions.

\noindent\textbf{Assumption 1.} \textit{The sign of input coefficient ${\mathop{\rm sgn}} (b)$ is known.}

\noindent\textbf{Assumption 2.} \textit{The reference signal $r$ and its first $\left\lceil {\sum\nolimits_{i = 1}^j {{\alpha _i}} } \right\rceil $-th order derivatives are piecewise continuous and bounded, $j=1, \cdots, n$.}

There are three widely accepted definitions of fractional order derivative ${{\mathscr D}^{{\alpha}}}$, but the Caputo's one is chosen for subsequent use due to some good  properties, namely its derivative of constant equals zero and there's no need to calculate fractional order derivative of initial value when operating Laplace transformation. The Caputo's definition is
\begin{eqnarray}\label{Eq2-3}
\textstyle
{}_c^{}{\mathscr D}_t^\alpha f\left( t \right) = \frac{1}{{\Gamma \left( {m - \alpha } \right)}}\int_c^t {\frac{{{f^{\left( m \right)}}\left( \tau  \right)}}{{{{\left( {t - \tau } \right)}^{\alpha  - m + 1}}}}{\rm{d}}\tau },
\end{eqnarray}
where $m-1< \alpha<m, m\in \mathbb{N}^+$, $\Gamma \left( \alpha  \right) = \int_0^\infty  {{x^{\alpha  - 1}}{{\rm e}^{ - x}}{\rm{d}}x} $ is the Gamma function.

On the basis of Caputo's definition, the next additive law of exponents holds on condition that $f\left( c \right)=0$ and $0<p,q<1$
\begin{eqnarray}\label{Eq2-4}
_c^{}{\mathscr D}_t^p\left( {_c^{}{\mathscr D}_t^qf\left( t \right)} \right) ={_c^{}{\mathscr D}_t^q\left( {_c^{}{\mathscr D}_t^pf\left( t \right)} \right)} = {}_c^{}{\mathscr D}_t^{p + q}f\left( t \right).
\end{eqnarray}

For purpose of convenient expression, the $\alpha$  order derivative from initial time zero ${}_0^{}{\mathscr D}_t^\alpha $ is simplified as ${\mathscr D}^\alpha $. Just before uncovering the main results, some helpful lemmas need referring.

\begin{lemma}\label{Lemma 1}
(see \cite{Montseny:1998LAAS}) The differential equation ${{\mathscr D}^{\alpha }}y\left(t \right)=u\left(t \right)$ with fractional order $0<\alpha<1$, ${y\left(t \right) \in \mathbb{R}}$ and ${u\left(t \right) \in \mathbb{R}}$ can be transformed into the following linear continuous frequency distributed model
\begin{eqnarray}\label{Eq2-5}
\left\{ {\begin{array}{rl}
\frac{{\partial z\left( {\omega ,t} \right)}}{{\partial t}} = & \hspace{-8pt} - \omega z\left( {\omega ,t} \right) + u\left( t \right),\\
y\left( t \right) = & \hspace{-8pt} \int_0^\infty  {{\mu _\alpha }\left( \omega  \right)z\left( {\omega ,t} \right){\rm{d}}\omega },
\end{array}} \right.
\end{eqnarray}
where ${{\mu _\alpha }\left( \omega  \right) = {{{\omega ^{ - \alpha }}\sin \left( {\alpha \pi } \right)} \mathord{\left/
 {\vphantom {{{\omega ^{ - \alpha }}\sin \left( {\alpha \pi } \right)} \pi }} \right.
 \kern-\nulldelimiterspace} \pi }}$ and $z\left( {\omega ,t} \right) \in \mathbb{R} $ is the true state of the system.
\end{lemma}

\begin{remark}\label{Remark 1} The Lemma 1 implies the infinite dimension of fractional order state. And this lemma derives from zero initial condition, while the system response by definition of Riemann-Liouville or Caputo corresponds to the response of frequency distributed model with specific initial value \cite{Trigeassou:2012SIVP}. With the help of equivalent frequency distributed model, the stability could be analyzed via Lyapunov technique where true sate replaces pseudo sate \cite{Trigeassou:2013CMA}.
\end{remark}

\begin{lemma}\label{Lemma 2}
(see \cite{Wei:2015JOTA}) Supposing $r_1$ is the allowed maximum of control input $u(t)$ and $v(t)$ is the input signal to be differentiated with $\left| {{{\mathscr D}^\alpha }v\left( t \right)} \right| < \infty $, the following fractional order tracking differentiator (FOTD)
\begin{eqnarray}\label{Eq2-6}
\left\{ \begin{array}{l}
{{\mathscr D}^\alpha }{x_1}\left( t \right) = {x_2}\left( t \right),\\
{{\mathscr D}^\alpha }{x_2}\left( t \right) = u\left( t \right),
\end{array} \right.
\end{eqnarray}
is convergent in the sense of ${x_1}(t)$ uniformly converging to $v(t)$ on $[0,\infty )$ as ${r_1} \to \infty $, and ${x_2}(t)$ is the general $\alpha $-th order differentiation of $v(t)$, where the control input $u\left( t \right) =  - {r_1}{\rm{tanh}}( {{x_1}\left( t \right) - v\left( t \right) + \frac{{{x_2}\left( t \right)\left| {{x_2}\left( t \right)} \right|}}{{{r_1}}}f\left( \alpha  \right),{r_2}} )$, ${r_2} > 0$, $f\left( \alpha  \right) = 1 - \frac{{{\Gamma ^2}\left( {\alpha  + 1} \right)}}{{\Gamma \left( {2\alpha  + 1} \right)}}$ and $\tanh \left( {z,\gamma } \right) = \frac{{{{\rm e}^{\gamma z}} - {{\rm e}^{ - \gamma z}}}}{{{{\rm e}^{\gamma z}} + {{\rm e}^{ - \gamma z}}}}$.
\end{lemma}

\begin{remark}\label{Remark 2}
In fact, the traditional integer order tracking differentiator could be regarded as a special case of the fractional order one \cite{Han:2009TIE}. Therefore, the fractional order derivative of signal will be obtained in time with required speed and smoothness by tuning parameters $r_1$ and $r_2$ of the FOTD  as needed.
\end{remark}

\section{Main Results}\label{Section 3}

A fractional order adaptive backstepping state feedback control method is unfolded in this section where two cases, known and unknown input coefficient, are studied separately. To deduce the main results smoothly, first of all, the controlled plant needs to undergo model transformation.


\subsection{Model transformation}
On the basis of following transformation
\begin{eqnarray}\label{Eq3-1}
\left\{ \begin{array}{l}
{{\bar x}_i} = {\delta _i}{x_i},i =1, \cdots, n,\\
{{\bar \psi }_i}\left( {{x_1}, \cdots ,{x_i}} \right) = {\delta _i}{\psi _i}\left( {{x_1}, \cdots ,{x_i}} \right),i =1, \cdots, n,\\
{{\bar \varphi }_i}\left( {{x_1}, \cdots ,{x_i}} \right) = {\delta _i}{\varphi _i}\left( {{x_1}, \cdots ,{x_i}} \right),i =1, \cdots, n,\\
{b'} = {\delta _n}b,\\
{d'} = {\delta _n}d,
\end{array} \right.
\end{eqnarray}
with ${\delta _1} = 1,~{\delta _j} = \prod\nolimits_{i = 1}^{j - 1} {d_i}~(j = 2, \cdots ,n) $. The controlled plant (\ref{Eq2-1}) will be changed into the system below
\begin{eqnarray}\label{Eq3-2}
\left\{ \begin{array}{rl}
{{\mathscr D}^{{\alpha _i}}}{{\bar x}_i} =& \hspace{-8pt} {{\bar x}_{i + 1}} + {{\bar \psi }_i}({x_1}, \cdots ,{x_i}) \\
& \hspace{-8pt} + \bar \varphi _i^{\rm T}({x_1}, \cdots ,{x_i})\theta , ~i = 1, \cdots ,n - 1,\\
{{\mathscr D}^{{\alpha _n}}}{{\bar x}_n} =& \hspace{-8pt} {b'}u(v) + {{\bar \psi }_n}(x) + \bar \varphi _n^{\rm T}(x)\theta + {d'},\\
y =& \hspace{-8pt} {{\bar x}_1},
\end{array} \right.
\end{eqnarray}
which is known as the normalized fractional order chain system.

\begin{remark}\label{Remark 3}
Considering the backstepping procedure, the general lower triangular FOS (\ref{Eq2-1}) needs converting into the normalized fractional order chain system (\ref{Eq3-2}) so that the FOABC method could be adopted. When the coefficients $d_i=1~(~i=1, \cdots, n-1~)$, the controlled plant (\ref{Eq2-1}) reduces to the system (\ref{Eq3-2}) correspondingly.
\end{remark}


\subsection{Input decomposition}
In order to deal with input saturation and dead zone, an intermediate variable $w$ is introduced between actual control input $u$ and desired control input $v$, namely, $u(v)=u[w(v)]$. Let $u(w)$ and $w(v)$ represent the dead zone nonlinearity and saturation nonlinearity, respectively, whereupon the previous control input restraint shown in Fig \ref{Fig 1} is projected onto two directions as illustrated in Fig \ref{Fig 2} where upper left and lower right represent input saturation and dead zone, respectively. It is obvious that though an intermediary is introduced, after the input $v$ mapped to $w$ and $w$ mapped to $u$, the final relationship between $v$ and $u$ has not changed. Therefore, the problem of control input restrained by input saturation and dead zone could be solved separately as shown in Fig \ref{Fig 3}.
\begin{figure}[htbp]
  \centering
  \includegraphics[width=0.45\textwidth]{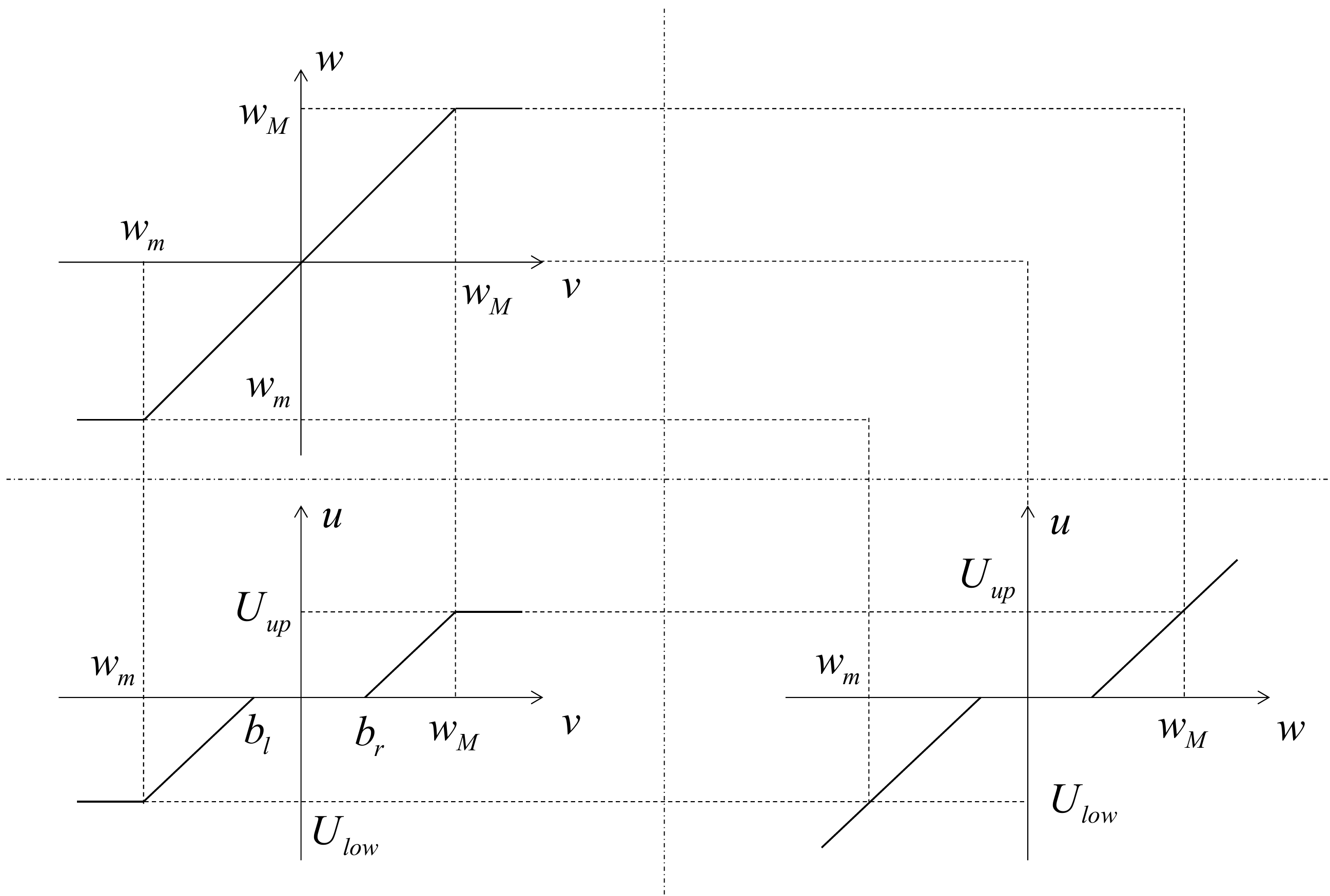}
  \caption{The decomposition of control input.}\label{Fig 2}
\end{figure}
\begin{figure}[htbp]
  \centering
  \includegraphics[width=0.3\textwidth]{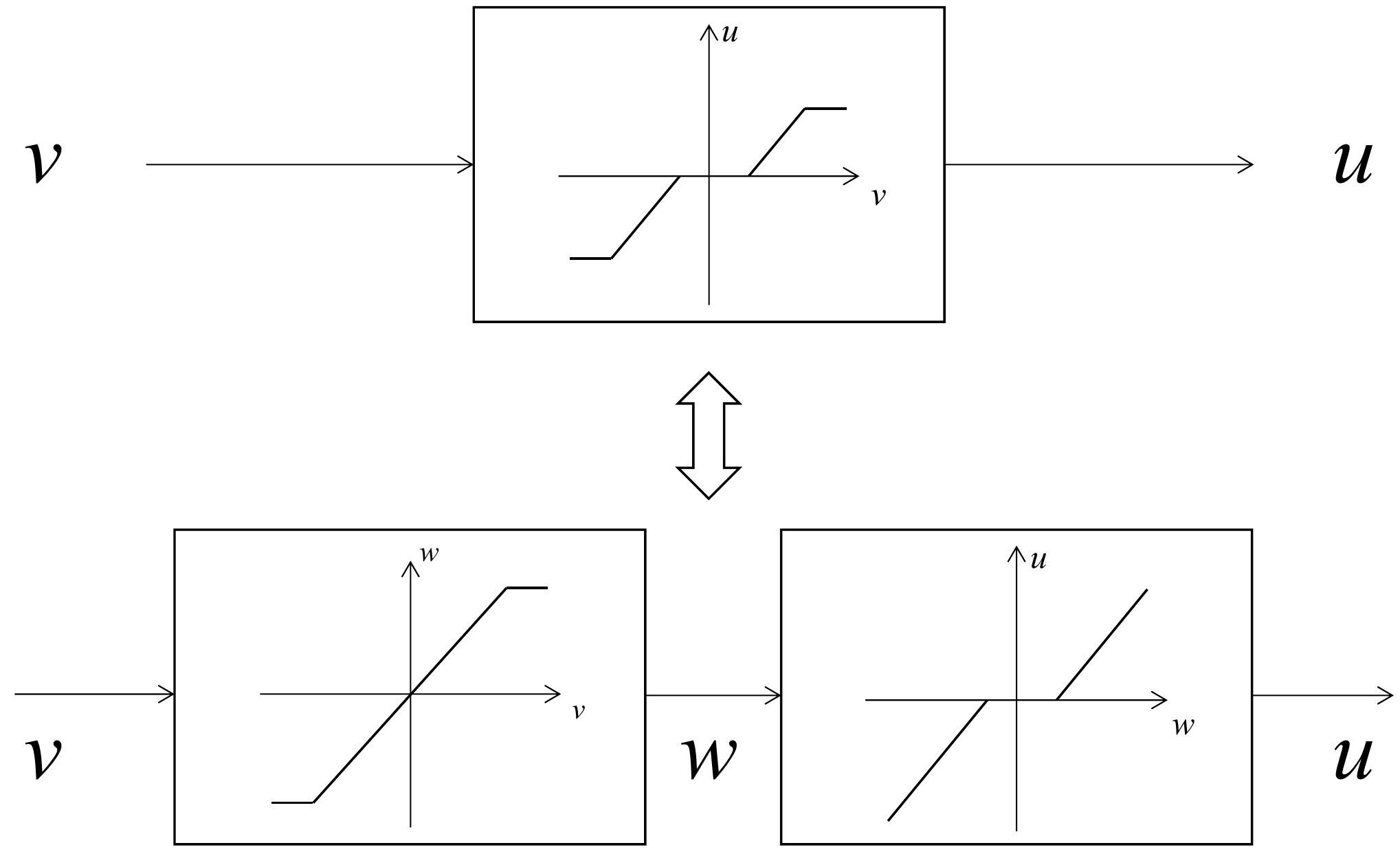}
  \caption{The intermediary $w$ between $v$ and $u$.}\label{Fig 3}
\end{figure}

The expression in (\ref{Eq2-2}) will be taken apart as follows
\begin{eqnarray}\label{Eq3-3}
w(v) = \left\{ \begin{array}{cl}
{w_M} &, ~v(t) \ge {w_M}\\
v(t) &, ~{w_m} \le v(t) < {w_M}\\
{w_m} &, ~v(t) < {w_m}
\end{array} \right.
\end{eqnarray}
\begin{eqnarray}\label{Eq3-4}
u(w) = \left\{ \begin{array}{cl}
m[w(v) - {b_r}] &, ~w(v) \ge {b_r}\\
0 &, ~{b_l} \le w(v) < {b_r}\\
m[w(v) - {b_l}] &, ~w(v) < {b_l}
\end{array} \right.
\end{eqnarray}

The dead zone part (\ref{Eq3-4}) could be rewritten as
\begin{eqnarray}\label{Eq3-5}
u(w) = mw(v) + d''(w),
\end{eqnarray}
where $d''(w)$ is a bounded term
\begin{eqnarray}\label{Eq3-6}
d''(w) = \left\{ \begin{array}{cl}
 - m{b_r} &, ~w \ge {b_r}\\
 - mw(v) &, ~{b_l} \le w < {b_r}\\
 - m{b_l} &, ~w < {b_l}
\end{array} \right.
\end{eqnarray}

Then the last state equation of (\ref{Eq3-2}) changes correspondingly
\begin{eqnarray}\label{Eq3-7}
\begin{array}{rl}
{{\mathscr D}^{{\alpha _n}}}{{\bar x}_n} =& \hspace{-8pt} b'u(v) + {{\bar \psi }_n}(x) + \bar \varphi _n^{\rm{T}}(x)\theta  + d'\\
 =& \hspace{-8pt} b'[ mw(v) + d''(w)]  + {{\bar \psi }_n}(x) + \bar \varphi _n^{\rm{T}}(x)\theta  + d'\\
 =& \hspace{-8pt} \bar bw(v) + {{\bar \psi }_n}(x) + \bar \varphi _n^{\rm{T}}(x)\theta  + \bar d,
\end{array}
\end{eqnarray}
where $\bar b = b'm$, and $\bar d = d' + b'd''(w)$ is a disturbance-like term within unknown bound $\left| {\bar d} \right| \le \bar D$. Finally, the system (\ref{Eq3-2}) will transform to
\begin{eqnarray}\label{Eq3-8}
\left\{ \begin{array}{rl}
{{\mathscr D}^{{\alpha _i}}}{{\bar x}_i} =& \hspace{-8pt} {{\bar x}_{i + 1}} + {{\bar \psi }_i}({x_1}, \cdots ,{x_i})   \\
& \hspace{-8pt} + \bar \varphi _i^{\rm T}({x_1}, \cdots ,{x_i})\theta , ~i = 1, \cdots ,n - 1,\\
{{\mathscr D}^{{\alpha _n}}}{{\bar x}_n} =& \hspace{-8pt} {\bar b}w(v) + {{\bar \psi }_n}(x) + \bar \varphi _n^{\rm T}(x)\theta + {\bar d},\\
y =& \hspace{-8pt} {{\bar x}_1}.
\end{array} \right.
\end{eqnarray}

The problem of dead zone and input saturation is transformed to the problem of input saturation and bounded disturbance at last.

It is noted that when either coefficient $b$ or slope $m$ is unknown, the parameter $\bar b$ will be unknown. In other words, the parameter $\bar b$ must be known under condition of known $b$ and $m$. In order to solve problem completely, both two cases will be studied afterwards.


\subsection{FOABC with known $\bar b$}
For the purpose of compensating saturation, a fractional order auxiliary system is designed to generate virtual signals $\lambda  = {[~{\lambda _1}, {\lambda _2}, \cdots , {\lambda _n}~]^{\rm T}}$ in the first place
\begin{eqnarray}\label{Eq3-9}
\left\{ \begin{array}{rl}
{{\mathscr D}^{{\alpha _i}}}{\lambda _i} =& \hspace{-8pt} {\lambda _{i + 1}} - {c_i}{\lambda _i},i = 1, \cdots ,n - 1,\\
{{\mathscr D}^{{\alpha _n}}}{\lambda _n} =& \hspace{-8pt} \bar b\Delta w - {c_n}{\lambda _n},
\end{array} \right.
\end{eqnarray}
where $\Delta w = w - v$, $c_i>1$ ($~i=2,3,\cdots,n-1$) and $c_1,~c_n>0.5$.

\begin{theorem}\label{Theorem 1}
Considering the plant (\ref{Eq2-1}) with known $\bar b$, there is a control method that consists of

\noindent the error variables
\begin{eqnarray}\label{Eq3-10}
\left\{ \begin{array}{l}
{\varepsilon _1} = {{\bar x}_1} - r - {\lambda _1},\\
{\varepsilon _i} = {{\bar x}_i} - {{\mathscr D}^{\sum\limits_{j = 1}^{i - 1} {{\alpha _j}} }}r - {\tau _{i - 1}} - {\lambda _i}, ~i = 2, \cdots ,n,
\end{array} \right.
\end{eqnarray}
the stabilizing functions
\begin{eqnarray}\label{Eq3-11}
\left\{ \begin{array}{rl}
{\tau _1} =& \hspace{-8pt} - {c_1}({{\bar x}_1} - r) - {{\bar \psi }_1} - \bar \varphi _1^{\rm T}\hat \theta, \\
{\tau _i} =& \hspace{-8pt} - {c_i}({{\bar x}_i} - {{\mathscr D}^{\sum\limits_{j = 1}^{i - 1} {{\alpha _j}} }}r - {\tau _{i - 1}}) + {{\mathscr D}^{{\alpha _i}}}{\tau _{i - 1}}   \\
& \hspace{-8pt} - {{\bar \psi }_i} - \bar \varphi _i^{\rm T}\hat \theta , ~i = 2, \cdots ,n - 1,
\end{array} \right.
\end{eqnarray}
the parameter update law
\begin{eqnarray}\label{Eq3-12}
\left\{ \begin{array}{rl}
{{\mathscr D}^\beta }\hat \theta  =& \hspace{-8pt} \Lambda \sum\limits_{j = 1}^n {{\varepsilon _j}{{\bar \varphi }_j}}, \\
{{\mathscr D}^\rho }\hat D =& \hspace{-8pt} \xi \left| {{\varepsilon _n}} \right|,
\end{array} \right.
\end{eqnarray}
the adaptive control law
\begin{eqnarray}\label{Eq3-13}
\begin{array}{rl}
v =& \hspace{-8pt} \frac{1}{{\bar b}}[ - {c_n}{\varepsilon _n} + {{\mathscr D}^{\sum\limits_{j = 1}^n {{\alpha _j}} }}r + {{\mathscr D}^{{\alpha _n}}}{\tau _{n - 1}} - {{\bar \psi }_n} - \bar \varphi _n^{\rm{T}}\hat \theta  \\
& \hspace{-8pt} - {\mathop{\rm sgn}} ({\varepsilon _n})\hat D - {c_n}{\lambda _n}].
\end{array}
\end{eqnarray}
Then all the signals in the closed-loop adaptive system are globally uniformly bounded, and the asymptotic tracking is achieved as
\begin{eqnarray}\label{Eq3-14}
\mathop {\lim }\limits_{t \to \infty } [y(t) - r(t)] = 0,
\end{eqnarray}
where $\hat \theta$ is the parameter estimate of $\theta$, $\hat D$ is the parameter estimate of $\bar D$, $\Lambda  \in {{\mathbb R}^{q \times q}}$ is a positive definite matrix, $0 < \beta, ~\rho  < 1$ and $\xi  >0$.
\end{theorem}


\begin{proof}
Because of the uncertain $\theta$, the control method must perform the function of controller as well as estimator. If the ${\hat \theta }$ is regarded as the estimate of $\theta $, the estimated error $\tilde \theta  = \theta  - \hat \theta $ naturally forms and then the following equation is obtained in view of Caputo's definition
\begin{eqnarray}\label{Eq3-15}
{{\mathscr D}^\beta }\tilde \theta  = {{\mathscr D}^\beta }\theta  - {{\mathscr D}^\beta }\hat \theta  =  - {{\mathscr D}^\beta }\hat \theta,
\end{eqnarray}
with the fractional order of update law $0< \beta < 1$. Based on Lemma \ref{Lemma 1}, (\ref{Eq3-15}) will be transformed into the frequency distributed model
\begin{eqnarray}\label{Eq3-16}
\left\{ \begin{array}{rl}
\frac{{\partial {z_\theta }(\omega ,t)}}{{\partial t}} =& \hspace{-8pt} - \omega {z_\theta }(\omega ,t) - {{\mathscr D}^\beta }\hat \theta ,\\
\tilde \theta  =& \hspace{-8pt} \int_0^\infty  {{\mu _\beta }(\omega ){z_\theta }(\omega ,t){\rm d}\omega },
\end{array} \right.
\end{eqnarray}
with ${z_\theta }(\omega ,t) \in {{\mathbb R}^q}$ and ${\mu _\beta }{\rm{(}}\omega {\rm{) = }}\frac{{\sin (\beta \pi )}}{{{\omega ^\beta }\pi }}$.


\noindent Step 1. Let's begin with the first equation in (\ref{Eq3-10}) by calculating the fractional order derivative and introducing virtual control $\tau_1$
\begin{eqnarray}\label{Eq3-17}
\begin{array}{rl}
{{\mathscr D}^{{\alpha _1}}}{\varepsilon _1} =& \hspace{-8pt}  {{\mathscr D}^{{\alpha _1}}}{{\bar x}_1} - {{\mathscr D}^{{\alpha _1}}}r - {{\mathscr D}^{{\alpha _1}}}{\lambda _1}  \\
 =& \hspace{-8pt}  {{\bar x}_2} + {{\bar \psi }_1} + \bar \varphi _1^{\rm T}\theta  - {{\mathscr D}^{{\alpha _1}}}r - {\lambda _2} + {c_1}{\lambda _1}  \\
 =& \hspace{-8pt} {\varepsilon _2} + {\tau _1} + {{\bar \psi }_1} + \bar \varphi _1^{\rm T}\theta  + {c_1}{\lambda _1}.
\end{array}
\end{eqnarray}

After transformation into the related frequency distributed model, the previous (\ref{Eq3-17}) will be
\begin{eqnarray}\label{Eq3-18}
\left\{ \begin{array}{rl}
\frac{{\partial {z_1}(\omega ,t)}}{{\partial t}} =& \hspace{-8pt}  - \omega {z_1}(\omega ,t) + {\varepsilon _2} + {\tau _1} \\
& \hspace{-8pt} + {{\bar \psi }_1} + \bar \varphi _1^{\rm T}\theta  + {c_1}{\lambda _1},\\
{\varepsilon _1} =& \hspace{-8pt} \int_0^\infty  {{\mu _{{\alpha _1}}}(\omega ){z_1}(\omega ,t){\rm d}\omega },
\end{array} \right.
\end{eqnarray}
with ${\mu _{{\alpha _1}}}{\rm{(}}\omega {\rm{) = }}\frac{{\sin ({\alpha _1}\pi )}}{{{\omega ^{{\alpha _1}}}\pi }}$.

Selecting the Lyapunov function as
\begin{eqnarray}\label{Eq3-19}
\begin{array}{rl}
{V_1} =& \hspace{-8pt} \frac{1}{2}\int_0^\infty  {{\mu _\beta }(\omega )z_\theta ^{\rm T}(\omega ,t){\Lambda ^{ - 1}}{z_\theta }(\omega ,t){\rm d}\omega }  \\
& \hspace{-8pt} + \frac{1}{2}\int_0^\infty  {{\mu _{{\alpha _1}}}(\omega )z_1^2(\omega ,t){\rm d}\omega },
\end{array}
\end{eqnarray}
then its derivative is expressed as
\begin{eqnarray}\label{Eq3-20}
\begin{array}{rl}
{\dot V_1} =& \hspace{-8pt}  - \int_0^\infty  {\omega {\mu _\beta }(\omega )z_\theta ^{\rm T}(\omega ,t){\Lambda ^{ - 1}}{z_\theta }(\omega ,t){\rm d}\omega }   \\
& \hspace{-8pt} - \int_0^\infty  {\omega {\mu _{{\alpha _1}}}(\omega )z_1^2(\omega ,t){\rm d}\omega } - {\tilde \theta ^{\rm T}}{\Lambda ^{ - 1}}{{\mathscr D}^\beta }\hat \theta \\
& \hspace{-8pt}  + {\varepsilon _1}({\varepsilon _2} + {\tau _1} + {\bar \psi _1} + \bar \varphi _1^{\rm T}\theta  + {c_1}{\lambda _1}).
\end{array}
\end{eqnarray}

Designing the stabilizing function ${\tau _1}$ as (\ref{Eq3-11}), the resulting derivative is
\begin{eqnarray}\label{Eq3-21}
\begin{array}{rl}
{{\dot V}_1} =& \hspace{-8pt}  - \int_0^\infty  {\omega {\mu _\beta }(\omega )z_\theta ^{\rm T}(\omega ,t){\Lambda ^{ - 1}}{z_\theta }(\omega ,t){\rm d}\omega } \\
& \hspace{-8pt} - \int_0^\infty  {\omega {\mu _{{\alpha _1}}}(\omega )z_1^2(\omega ,t){\rm d}\omega } \\
& \hspace{-8pt} + {{\tilde \theta }^{\rm T}}({\varepsilon _1}{{\bar \varphi }_1} - {\Lambda ^{ - 1}}{{\mathscr D}^\beta }\hat \theta ) - {c_1}\varepsilon _1^2 + {\varepsilon _1}{\varepsilon _2}\\
\le & \hspace{-8pt}  - \int_0^\infty  {\omega {\mu _\beta }(\omega )z_\theta ^{\rm T}(\omega ,t){\Lambda ^{ - 1}}{z_\theta }(\omega ,t){\rm d}\omega } \\
& \hspace{-8pt} - \int_0^\infty  {\omega {\mu _{{\alpha _1}}}(\omega )z_1^2(\omega ,t){\rm d}\omega } \\
& \hspace{-8pt} +{{\tilde \theta }^{\rm T}}({\varepsilon _1}{{\bar \varphi }_1} - {\Lambda ^{ - 1}}{{\mathscr D}^\beta }\hat \theta ) - {{\bar c}_1}\varepsilon _1^2{\rm{ + }}\frac{1}{2}\varepsilon _2^2,
\end{array}
\end{eqnarray}
where ${\bar c_1} = {c_1} - \frac{1}{2} > 0$.

If ${\varepsilon _2} = 0$ and ${{\mathscr D}^\beta }\hat \theta  = \Lambda {\varepsilon _1}{{\bar \varphi }_1}$, then ${\varepsilon _1}$ and ${\tilde \theta }$ are both asymptotically convergent to zero in accordance with LaSalle invariance principle \cite{LaSalle:1965DEDS}.


\noindent Step $i (i=2,\cdots,n-1).$
We continue to investigate the $i$-th equation of (\ref{Eq3-10}) with the introduction of virtual control variable $\tau_i$, then the differential of $\varepsilon _i$ is
\begin{eqnarray}\label{Eq3-22}
\begin{array}{rl}
{{\mathscr D}^{{\alpha _i}}}{\varepsilon _i} =& \hspace{-8pt} {{\mathscr D}^{{\alpha _i}}}{{\bar x}_i} - {{\mathscr D}^{\sum\limits_{j = 1}^i {{\alpha _j}} }}r - {{\mathscr D}^{{\alpha _i}}}{\tau _{i - 1}} - {{\mathscr D}^{{\alpha _i}}}{\lambda _i}\\
 =& \hspace{-8pt} {{\bar x}_{i + 1}} + {{\bar \psi }_i} + \bar \varphi _i^{\rm T}\theta  - {{\mathscr D}^{\sum\limits_{j = 1}^i {{\alpha _j}} }}r - {{\mathscr D}^{{\alpha _i}}}{\tau _{i - 1}} \\
& \hspace{-8pt}  - {\lambda _{i + 1}} + {c_i}{\lambda _i} \\
 =& \hspace{-8pt} {\varepsilon _{i + 1}} + {\tau _i} + {{\bar \psi }_i} + \bar \varphi _i^{\rm T}\theta  - {{\mathscr D}^{{\alpha _i}}}{\tau _{i - 1}} + {c_i}{\lambda _i}.
\end{array}
\end{eqnarray}

Its frequency distributed model is shown below
\begin{eqnarray}\label{Eq3-23}
\left\{ \begin{array}{rl}
\frac{{\partial {z_i}(\omega ,t)}}{{\partial t}} =& \hspace{-8pt}  - \omega {z_i}(\omega ,t) + {\varepsilon _{i + 1}} + {\tau _i}  + {{\bar \psi }_i} + \bar \varphi _i^{\rm T}\theta  \\
& \hspace{-8pt} - {{\mathscr D}^{{\alpha _i}}}{\tau _{i - 1}} + {c_i}{\lambda _i},\\
{\varepsilon _i} =& \hspace{-8pt} \int_0^\infty  {{\mu _{{\alpha _i}}}(\omega ){z_i}(\omega ,t){\rm d}\omega },
\end{array} \right.
\end{eqnarray}
with ${\mu _{{\alpha _i}}}{\rm{(}}\omega {\rm{) = }}\frac{{\sin ({\alpha _i}\pi )}}{{{\omega ^{{\alpha _i}}}\pi }}$.

This step aims at stabilizing the system (\ref{Eq3-23}) through Lyapunov function below
\begin{eqnarray}\label{Eq3-24}
{V_i} = {V_{i-1}} + \frac{1}{2}\int_0^\infty  {{\mu _{{\alpha _i}}}\left( \omega  \right)z_i^2\left( {\omega ,t} \right){\rm{d}}\omega } .
\end{eqnarray}

By calculating the derivative of $V_i$
\begin{eqnarray}\label{Eq3-25}
\begin{array}{rl}
{{\dot V}_i} =& \hspace{-8pt} {{\dot V}_{i - 1}} - \int_0^\infty  {\omega {\mu _{{\alpha _i}}}(\omega )z_i^2(\omega ,t){\rm d}\omega } \\
& \hspace{-8pt} + {\varepsilon _i}({\varepsilon _{i + 1}} + {\tau _i} + {{\bar \psi }_i} + \bar \varphi _i^{\rm T}\theta  - {{\mathscr D}^{{\alpha _i}}}{\tau _{i - 1}} + {c_i}{\lambda _i})\\
 \le & \hspace{-8pt} - \int_0^\infty  {\omega {\mu _\beta }(\omega )z_\theta ^{\rm T}(\omega ,t){\Lambda ^{ - 1}}{z_\theta }(\omega ,t){\rm d}\omega } \\
& \hspace{-8pt} - \sum\limits_{j = 1}^i {\int_0^\infty  {\omega {\mu _{{\alpha _j}}}(\omega )z_j^2(\omega ,t){\rm d}\omega } } \\
& \hspace{-8pt} + {{\tilde \theta }^{\rm T}}(\sum\limits_{j = 1}^{i - 1} {{\varepsilon _j}{{\bar \varphi }_j}}  - {\Lambda ^{ - 1}}{{\mathscr D}^\beta }\hat \theta ) - \sum\limits_{j = 1}^{i - 1} {{{\bar c}_j}\varepsilon _j^2}  + \frac{1}{2}\varepsilon _i^2 \\
& \hspace{-8pt} + {\varepsilon _i}({\varepsilon _{i + 1}} + {\tau _i} + {{\bar \psi }_i} + \bar \varphi _i^{\rm T}\theta  - {{\mathscr D}^{{\alpha _i}}}{\tau _{i - 1}} + {c_i}{\lambda _i}),
\end{array}
\end{eqnarray}
and designing the stabilizing function $\tau_i$ as (\ref{Eq3-11}), we further infer the derivative $\dot V_i$
\begin{eqnarray}\label{Eq3-26}
\begin{array}{rl}
{{\dot V}_i} \le & \hspace{-8pt}  - \int_0^\infty  {\omega {\mu _\beta }(\omega )z_\theta ^{\rm T}(\omega ,t){\Lambda ^{ - 1}}{z_\theta }(\omega ,t){\rm d}\omega }  \\
& \hspace{-8pt} - \sum\limits_{j = 1}^i {\int_0^\infty  {\omega {\mu _{{\alpha _j}}}(\omega )z_j^2(\omega ,t){\rm d}\omega } } \\
& \hspace{-8pt} + {{\tilde \theta }^{\rm T}}(\sum\limits_{j = 1}^i {{\varepsilon _j}{{\bar \varphi }_j}}  - {\Lambda ^{ - 1}}{{\mathscr D}^\beta }\hat \theta ) - \sum\limits_{j = 1}^{i - 1} {{{\bar c}_j}\varepsilon _j^2}  \\
& \hspace{-8pt} + \frac{1}{2}\varepsilon _i^2  - {c_i}\varepsilon _i^2 + {\varepsilon _i}{\varepsilon _{i + 1}}\\
\le & \hspace{-8pt} - \int_0^\infty  {\omega {\mu _\beta }(\omega )z_\theta ^{\rm T}(\omega ,t){\Lambda ^{ - 1}}{z_\theta }(\omega ,t){\rm d}\omega } \\
& \hspace{-8pt}  - \sum\limits_{j = 1}^i {\int_0^\infty  {\omega {\mu _{{\alpha _j}}}(\omega )z_j^2(\omega ,t){\rm d}\omega } } \\
& \hspace{-8pt} + {{\tilde \theta }^{\rm T}}(\sum\limits_{j = 1}^i {{\varepsilon _j}{{\bar \varphi }_j}}  - {\Lambda ^{ - 1}}{{\mathscr D}^\beta }\hat \theta ) - \sum\limits_{j = 1}^i {{{\bar c}_j}\varepsilon _j^2}  + \frac{1}{2}\varepsilon _{i + 1}^2,
\end{array}
\end{eqnarray}
where ${{\bar c}_i} = {c_i} - 1 > 0$.

Indeed, when $\varepsilon_{i+1}=0$ and ${{\mathscr D}^\beta }\hat \theta  = \Lambda \sum\limits_{j = 1}^i {{\varepsilon _j}{{\bar \varphi }_j}} $, then $\varepsilon_i$ is asymptotically convergent to zero. While $\varepsilon_i=0$, the update law $\Lambda \sum\limits_{j = 1}^i {{\varepsilon _j}{{\bar \varphi }_j}}$ reduces to $\Lambda \sum\limits_{j = 1}^{i{\rm{ - }}1} {{\varepsilon _j}{{\bar \varphi }_j}} $, which returns to step $i-1.$


\noindent Step $n.$
It is not until the last step that the real control input $v$ finally turns up, hence the whole system will be under control as adaptive control law design finishes. Before the last design, however, there is a disturbance-like term with unknown bound $\bar D$ worthy of attention. If the ${\hat D }$ is regarded as the estimate of $\bar D $, the estimated error will be $\tilde D  = \bar D  - \hat D$, and then the following equation is obtained in view of Caputo's definition
\begin{eqnarray}\label{Eq3-27}
{{\mathscr D}^\rho }\tilde D  = {{\mathscr D}^\rho }\bar D  - {{\mathscr D}^\rho }\hat D  =  - {{\mathscr D}^\rho }\hat D,
\end{eqnarray}
with the fractional order of update law $0< \rho < 1$. Based on Lemma \ref{Lemma 1}, (\ref{Eq3-27}) will be transformed into the frequency distributed model
\begin{eqnarray}\label{Eq3-28}
\left\{ \begin{array}{rl}
\frac{{\partial {z_D}(\omega ,t)}}{{\partial t}} =& \hspace{-8pt} - \omega {z_D}(\omega ,t) - {{\mathscr D}^\rho }\hat D,\\
\tilde D =& \hspace{-8pt} \int_0^\infty  {{\mu _\rho }(\omega ){z_D}(\omega ,t){\rm{d}}\omega },
\end{array} \right.
\end{eqnarray}
with ${\mu _\rho }{\rm{(}}\omega {\rm{) = }}\frac{{\sin (\rho \pi )}}{{{\omega ^\rho }\pi }}$.

Similar to the previous steps, the error variable $\varepsilon_n$ will be derived based on the equation (\ref{Eq3-10}), subsequently its differential function comes out
\begin{eqnarray}\label{Eq3-29}
\begin{array}{rl}
{{\mathscr D}^{{\alpha _n}}}{\varepsilon _n} =& \hspace{-8pt} {{\mathscr D}^{{\alpha _n}}}{{\bar x}_n} - {{\mathscr D}^{\sum\limits_{j = 1}^n {{\alpha _j}} }}r - {{\mathscr D}^{{\alpha _n}}}{\tau _{n - 1}} - {{\mathscr D}^{{\alpha _n}}}{\lambda _n}\\
 =& \hspace{-8pt} \bar bw(v) + {{\bar \psi }_n} + \bar \varphi _n^{\rm{T}}\theta  + \bar d - {{\mathscr D}^{\sum\limits_{j = 1}^n {{\alpha _j}} }}r \\
& \hspace{-8pt} - {{\mathscr D}^{{\alpha _n}}}{\tau _{n - 1}} - \bar b\Delta w + {c_n}{\lambda _n}\\
 =& \hspace{-8pt} \bar bv + {{\bar \psi }_n} + \bar \varphi _n^{\rm{T}}\theta  + \bar d - {{\mathscr D}^{\sum\limits_{j = 1}^n {{\alpha _j}} }}r \\
& \hspace{-8pt}  - {{\mathscr D}^{{\alpha _n}}}{\tau _{n - 1}} + {c_n}{\lambda _n}.
\end{array}
\end{eqnarray}

The frequency distributed model is
\begin{eqnarray}\label{Eq3-30}
\left\{ \begin{array}{rl}
\frac{{\partial {z_n}(\omega ,t)}}{{\partial t}} =& \hspace{-8pt} - \omega {z_n}(\omega ,t) + \bar bv + {{\bar \psi }_n} + \bar \varphi _n^{\rm{T}}\theta  \\
& \hspace{-8pt} + \bar d - {{\mathscr D}^{\sum\limits_{j = 1}^n {{\alpha _j}} }}r - {{\mathscr D}^{{\alpha _n}}}{\tau _{n - 1}} + {c_n}{\lambda _n},\\
{\varepsilon _n} =& \hspace{-8pt} \int_0^\infty  {{\mu _{{\alpha _n}}}(\omega ){z_n}(\omega ,t){\rm{d}}\omega },
\end{array} \right.
\end{eqnarray}
with ${\mu _{{\alpha _n}}}{\rm{(}}\omega {\rm{) = }}\frac{{\sin ({\alpha _n}\pi )}}{{{\omega ^{{\alpha _n}}}\pi }}$.

With the purpose of ensuring ${\varepsilon_n} \to 0$ as $t \to \infty$, the last and overall Lyapunov function is decided
\begin{eqnarray}\label{Eq3-31}
\begin{array}{rl}
{V_n} =& \hspace{-8pt} {V_{n - 1}} + \frac{1}{2}\int_0^\infty  {{\mu _{{\alpha _n}}}(\omega )z_n^2(\omega ,t){\rm d}\omega } \\
& \hspace{-8pt} + \frac{1}{{2\xi }}\int_0^\infty  {{\mu _\rho }(\omega )z_D^2(\omega ,t){\rm{d}}\omega }.
\end{array}
\end{eqnarray}

By adopting the inequality we induce in the step $n-1$, the derivative of (\ref{Eq3-31}) is
\begin{eqnarray}\label{Eq3-32}
\begin{array}{rl}
{{\dot V}_n} =& \hspace{-8pt} {{\dot V}_{n - 1}} - \int_0^\infty  {\omega {\mu _{{\alpha _n}}}(\omega )z_n^2(\omega ,t){\rm{d}}\omega } \\
& \hspace{-8pt} - \frac{1}{\xi }\int_0^\infty  {\omega {\mu _\rho }(\omega )z_D^2(\omega ,t){\rm{d}}\omega } + {\varepsilon _n}(\bar bv + {{\bar \psi }_n} \\
& \hspace{-8pt}  + \bar \varphi _n^{\rm{T}}\theta  + \bar d - {{\mathscr D}^{\sum\limits_{j = 1}^n {{\alpha _j}} }}r - {{\mathscr D}^{{\alpha _n}}}{\tau _{n - 1}} + {c_n}{\lambda _n}) \\
& \hspace{-8pt} - \frac{1}{\xi }\tilde D{{\mathscr D}^\rho }\hat D \\
 \le & \hspace{-8pt} - \int_0^\infty  {\omega {\mu _\beta }(\omega )z_\theta ^{\rm{T}}(\omega ,t){\Lambda ^{ - 1}}{z_\theta }(\omega ,t){\rm{d}}\omega } \\
& \hspace{-8pt}   - \sum\limits_{j = 1}^n {\int_0^\infty  {\omega {\mu _{{\alpha _j}}}(\omega )z_j^2(\omega ,t){\rm{d}}\omega } } \\
& \hspace{-8pt} - \frac{1}{\xi }\int_0^\infty  {\omega {\mu _\rho }(\omega )z_D^2(\omega ,t){\rm{d}}\omega } - \sum\limits_{j = 1}^{n - 1} {{{\bar c}_j}\varepsilon _j^2}  + \frac{1}{2}\varepsilon _n^2   \\
& \hspace{-8pt} + {{\tilde \theta }^{\rm{T}}}(\sum\limits_{j = 1}^{n - 1} {{\varepsilon _j}{{\bar \varphi }_j}} - {\Lambda ^{ - 1}}{{\mathscr D}^\beta }\hat \theta ) + {\varepsilon _n}(\bar bv + {{\bar \psi }_n}   \\
& \hspace{-8pt} + \bar \varphi _n^{\rm{T}}\theta + \bar d - {{\mathscr D}^{\sum\limits_{j = 1}^n {{\alpha _j}} }}r   - {{\mathscr D}^{{\alpha _n}}}{\tau _{n - 1}} + {c_n}{\lambda _n}) \\
& \hspace{-8pt} - \frac{1}{\xi }\tilde D{{\mathscr D}^\rho }\hat D.
\end{array}
\end{eqnarray}

Designing the adaptive control law $v$ as (\ref{Eq3-13}), the simplified derivative $\dot V_n$ is
\begin{eqnarray}\label{Eq3-33}
\begin{array}{rl}
{{\dot V}_n} \le & \hspace{-8pt} - \int_0^\infty  {\omega {\mu _\beta }(\omega )z_\theta ^{\rm{T}}(\omega ,t){\Lambda ^{ - 1}}{z_\theta }(\omega ,t){\rm{d}}\omega }  \\
& \hspace{-8pt} - \sum\limits_{j = 1}^n {\int_0^\infty  {\omega {\mu _{{\alpha _j}}}(\omega )z_j^2(\omega ,t){\rm{d}}\omega } } \\
& \hspace{-8pt} - \frac{1}{\xi }\int_0^\infty  {\omega {\mu _\rho }(\omega )z_D^2(\omega ,t){\rm{d}}\omega } - \sum\limits_{j = 1}^{n - 1} {{{\bar c}_j}\varepsilon _j^2}  + \frac{1}{2}\varepsilon _n^2\\
& \hspace{-8pt} + {{\tilde \theta }^{\rm{T}}}(\sum\limits_{j = 1}^n {{\varepsilon _j}{{\bar \varphi }_j}}  - {\Lambda ^{ - 1}}{{\mathscr D}^\beta }\hat \theta ) - {c_n}\varepsilon _n^2 + {\varepsilon _n}\bar d \\
& \hspace{-8pt} - \left| {{\varepsilon _n}} \right|\hat D - \frac{1}{\xi }\tilde D{{\mathscr D}^\rho }\hat D.
\end{array}
\end{eqnarray}

Because of the following inequality
\begin{eqnarray}\label{Eq3-34}
\begin{array}{rl}
{\varepsilon _n}\bar d - \left| {{\varepsilon _n}} \right|\hat D  \le  \left| {{\varepsilon _n}} \right|\bar D - \left| {{\varepsilon _n}} \right|\hat D  = \tilde D\left| {{\varepsilon _n}} \right|,
\end{array}
\end{eqnarray}
and the the parameter update law ${{\mathscr D}^\beta }\hat \theta, ~{{\mathscr D}^\rho }\hat D$ (\ref{Eq3-12}), the $\dot V_n$ is
\begin{eqnarray}\label{Eq3-35}
\begin{array}{rl}
\dot V_n \le & \hspace{-8pt} -\int_0^\infty  {\omega {\mu _\beta }(\omega )z_\theta ^{\rm{T}}(\omega ,t){\Lambda ^{ - 1}}{z_\theta }(\omega ,t){\rm{d}}\omega }  \\
& \hspace{-8pt} - \sum\limits_{j = 1}^n {\int_0^\infty  {\omega {\mu _{{\alpha _j}}}(\omega )z_j^2(\omega ,t){\rm{d}}\omega } } \\
& \hspace{-8pt} - \frac{1}{\xi }\int_0^\infty  {\omega {\mu _\rho }(\omega )z_D^2(\omega ,t){\rm{d}}\omega } - \sum\limits_{j = 1}^n {{{\bar c}_j}\varepsilon _j^2},
\end{array}
\end{eqnarray}
where ${{\bar c}_n} = {c_n} - \frac{1}{2} > 0$.

According to the LaSalle invariant principle, the $z_i(\omega,t)$ are convergent to the zero point, which makes the error variables $\varepsilon_i$ convergent to zero. Moreover, as the input of constructed system (\ref{Eq3-9}), the error $\Delta w$ has nothing to do with error variables $\varepsilon_i$, thus the controller design will not be influenced by $\Delta w$ and the boundedness of estimation will be guaranteed \cite{Lin:2012NODY}.
\end{proof}


In the recursive procedure of controllers design, every coefficient $c_i$ is required to be greater than a certain constant with the purpose of establishing inequality and eliminating the product ${\varepsilon_i}{\varepsilon_{i+1}}$. This procedure leads to the conservatism of application, though relatively superior control performance will be realized afterwards. To meet the different needs of practice, a more general and flexible FOABC strategy is derived from Theorem 1.

\begin{corollary}\label{Corollary 1}
For the purpose of compensation, a fractional order auxiliary system is designed to generate virtual signals $\lambda  = {[~{\lambda _1}, {\lambda _2}, \cdots , {\lambda _n}~]^{\rm T}}$
\begin{eqnarray}\label{Eq3-36}
\left\{ \begin{array}{rl}
{{\mathscr D}^{{\alpha _i}}}{\lambda _i} =& \hspace{-8pt} {\lambda _{i + 1}} - {a_i}{\mathop{\rm sgn}} ({\lambda _i}){\left| {{\lambda _i}} \right|^{{\mu _i}}}, i = 1, \cdots ,n - 1,\\
{{\mathscr D}^{{\alpha _n}}}{\lambda _n} =& \hspace{-8pt} \bar b\Delta w - {a_n}{\mathop{\rm sgn}} ({\lambda _n}){\left| {{\lambda _n}} \right|^{{\mu _n}}}.
\end{array} \right.
\end{eqnarray}
There is a control method that consists of

\noindent the error variables
\begin{eqnarray}\label{Eq3-37}
\left\{ \begin{array}{l}
{\varepsilon _1} = {{\bar x}_1} - r - {\lambda _1},\\
{\varepsilon _i} = {{\bar x}_i} - {{\mathscr D}^{\sum\limits_{j = 1}^{i - 1} {{\alpha _j}} }}r - {\tau _{i - 1}} - {\lambda _i}, ~i = 2, \cdots ,n,
\end{array} \right.
\end{eqnarray}
the stabilizing functions
\begin{eqnarray}\label{Eq3-38}
\left\{ \begin{array}{rl}
{\tau _1} =& \hspace{-8pt} - {c_1}{\mathop{\rm sgn}} ({\varepsilon _1}){\left| {{\varepsilon _1}} \right|^{{\sigma _1}}} - {{\bar \psi }_1}  - \bar \varphi _1^{\rm T}\hat \theta  \\
& \hspace{-8pt} - {a_1}{\mathop{\rm sgn}} ({\lambda _1}){\left| {{\lambda _1}} \right|^{{\mu _1}}},\\
{\tau _i} = & \hspace{-8pt} - {\varepsilon _{i - 1}} - {c_i}{\mathop{\rm sgn}} ({\varepsilon _i}){\left| {{\varepsilon _i}} \right|^{{\sigma _i}}} - {{\bar \psi }_i} - \bar \varphi _i^{\rm T}\hat \theta   \\
& \hspace{-8pt} + {{\mathscr D}^{{\alpha _i}}}{\tau _{i - 1}} - {a_i}{\mathop{\rm sgn}} ({\lambda _i}){\left| {{\lambda _i}} \right|^{{\mu _i}}},\\
& \hspace{-8pt}i = 2, \cdots ,n - 1,
\end{array} \right.
\end{eqnarray}
the parameter update law
\begin{eqnarray}\label{Eq3-39}
\left\{ \begin{array}{rl}
{{\mathscr D}^\beta }\hat \theta  =& \hspace{-8pt} \Lambda \sum\limits_{j = 1}^n {{\varepsilon _j}{{\bar \varphi }_j}}, \\
{{\mathscr D}^\rho }\hat D =& \hspace{-8pt} \xi \left| {{\varepsilon _n}} \right|,
\end{array} \right.
\end{eqnarray}
the adaptive control law
\begin{eqnarray}\label{Eq3-40}
\begin{array}{rl}
v =& \hspace{-8pt} \frac{1}{{\bar b}}[ - {\varepsilon _{n - 1}} - {c_n}{\mathop{\rm sgn}} ({\varepsilon _n}){\left| {{\varepsilon _n}} \right|^{{\sigma _n}}} - {{\bar \psi }_n} - \bar \varphi _n^{\rm{T}}\hat \theta  \\
& \hspace{-8pt} - {\mathop{\rm sgn}} ({\varepsilon _n})\hat D + {{\mathscr D}^{\sum\limits_{j = 1}^n {{\alpha _j}} }}r + {{\mathscr D}^{{\alpha _n}}}{\tau _{n - 1}} \\
& \hspace{-8pt} - {a_n}{\mathop{\rm sgn}} ({\lambda _n}){\left| {{\lambda _n}} \right|^{{\mu _n}}}].
\end{array}
\end{eqnarray}
Then all the signals in the closed-loop adaptive system are globally uniformly bounded, and the asymptotic tracking is achieved as
\begin{eqnarray}\label{Eq3-41}
\mathop {\lim }\limits_{t \to \infty } [y(t) - r(t)] = 0,
\end{eqnarray}
where $\hat \theta$ is the parameter estimate of $\theta$, $\hat D$ is the parameter estimate of $\bar D$, $\Lambda  \in {{\mathbb R}^{q \times q}}$ is a positive definite matrix,  $ a_i , c_i , \xi>0$, $0<\rho , \beta , \alpha _i , \sigma_i , \mu_i<1 ~(i =1, \cdots, n)$ and $\Delta w = w - v$.
\end{corollary}

The proof of Corollary 1 is omitted here, because one can easily complete it according to the procedure of Theorem \ref{Theorem 1}.


\subsection{FOABC with unknown $\bar b$}
Last subsection discusses about the FOABC under the condition of known $\bar b$, while the unknown case is common and  more complicated. For example, when the slope $m$ or input coefficient $b$ is unknown, $\bar b=\delta_n b m$ will be unknown as well and an alternative FOABC method without known $\bar b$ is expected.

For the purpose of compensation, a fractional order auxiliary system is designed to generate virtual signals $\lambda  = {[~{\lambda _1}, {\lambda _2}, \cdots , {\lambda _n}~]^{\rm T}}$ in the first place
\begin{eqnarray}\label{Eq3-4-1}
\left\{ \begin{array}{l}
{{\mathscr D}^{{\alpha _i}}}{\lambda _i} = {\lambda _{i + 1}} - {c_i}{\lambda _i},i = 1, \cdots ,n - 1,\\
{{\mathscr D}^{{\alpha _n}}}{\lambda _n} = \frac{1}{{\hat p}}\Delta w - {c_n}{\lambda _n},
\end{array} \right.
\end{eqnarray}
where ${\hat p}$ is the estimate of $\frac{1}{{\bar b}}$, $\Delta w = w - v$, $c_i>1$ ($~i=2,3,\cdots,n-1$) and $c_1,~c_n>0.5$.

\begin{theorem}\label{Theorem 2}
Considering the plant (\ref{Eq2-1}) with known ${\rm sgn}(b)$, there is a control method that consists of

\noindent the error variables
\begin{eqnarray}\label{Eq3-4-2}
\left\{ \begin{array}{l}
{\varepsilon _1} = {{\bar x}_1} - r - {\lambda _1},\\
{\varepsilon _i} = {{\bar x}_i} - {{\mathscr D}^{\sum\limits_{j = 1}^{i - 1} {{\alpha _j}} }}r - {\tau _{i - 1}} - {\lambda _i},~i = 2, \cdots ,n,
\end{array} \right.
\end{eqnarray}
the stabilizing functions
\begin{eqnarray}\label{Eq3-4-3}
\left\{ \begin{array}{rl}
{\tau _1} =& \hspace{-8pt} - {c_1}({{\bar x}_1} - r) - {{\bar \psi }_1} - \bar \varphi _1^{\rm T}\hat \theta, \\
{\tau _i} =& \hspace{-8pt} - {c_i}({{\bar x}_i} - {{\mathscr D}^{\sum\limits_{j = 1}^{i - 1} {{\alpha _j}} }}r - {\tau _{i - 1}}) + {{\mathscr D}^{{\alpha _i}}}{\tau _{i - 1}} - {{\bar \psi }_i}  \\
& \hspace{-8pt} - \bar \varphi _i^{\rm T}\hat \theta, ~i = 2, \cdots ,n - 1,
\end{array} \right.
\end{eqnarray}
the parameter update law
\begin{eqnarray}\label{Eq3-4-4}
\left\{ \begin{array}{l}
{{\mathscr D}^\beta }\hat \theta  = \Lambda \sum\limits_{j = 1}^n {{\varepsilon _j}{{\bar \varphi }_j}}, \\
{{\mathscr D}^\gamma }\hat p =  - \eta {\mathop{\rm sgn}} (\bar b){\varepsilon _n}\bar w,\\
{{\mathscr D}^\rho }\hat D = \xi \left| {{\varepsilon _n}} \right|,
\end{array} \right.
\end{eqnarray}
the adaptive control law
\begin{eqnarray}\label{Eq3-4-5}
\left\{ \begin{array}{rl}
v =& \hspace{-8pt} \hat p\bar v,\\
\bar v =& \hspace{-8pt} - {c_n}{\lambda _n} - {c_n}{\varepsilon _n} - {\mathop{\rm sgn}} ({\varepsilon _n})\hat D + {{\mathscr D}^{\sum\limits_{j = 1}^n {{\alpha _j}} }}r \\
& \hspace{-8pt} + {{\mathscr D}^{{\alpha _n}}}{\tau _{n - 1}} - {{\bar \psi }_n} - {{\bar \varphi }^{\rm T}_n}\hat \theta .
\end{array} \right.
\end{eqnarray}
Then all the signals in the closed-loop adaptive system are globally uniformly bounded, and the asymptotic tracking is achieved as
\begin{eqnarray}\label{Eq3-4-6}
\mathop {\lim }\limits_{t \to \infty } [y(t) - r(t)] = 0,
\end{eqnarray}
where $\hat p$ is the parameter estimate of $p=\frac{1}{{\bar b}}$, $\hat \theta$ is the parameter estimate of $\theta$, $\hat D$ is the parameter estimate of $\bar D$, $\Lambda  \in {{\mathbb R}^{q \times q}}$ is a positive definite matrix, $0 < \gamma, ~ \beta, ~ \rho  < 1$, $\eta, \xi  > 0$ and $w = \hat p\bar w$.
\end{theorem}


\begin{proof}
Since the difference between known $\bar b$ and unknown $\bar b$ only depends on the last state equation (\ref{Eq2-1}), the first $n-1$ steps in this theorem are identical to the proof of Theorem \ref{Theorem 1}, thus they are omitted here in case of repetitive work. Due to the introduction of an unknown parameter $\bar b$, we first design a new estimator that regards $\hat p$  as the estimate of $p = 1 / \bar b$ and ${\tilde p}=p-\hat p$ as the relevant estimated error. Note that
\begin{eqnarray}\label{Eq3-4-7}
{{\mathscr D}^\gamma }\tilde p = {{\mathscr D}^\gamma }p - {{\mathscr D}^\gamma }\hat p =  - {{\mathscr D}^\gamma }\hat p,
\end{eqnarray}
with the fractional order of update law $0< \gamma < 1$. Its frequency distributed model is analogous to (\ref{Eq3-16}) and (\ref{Eq3-28}) which could be written as
\begin{eqnarray}\label{Eq3-4-8}
\left\{ \begin{array}{rl}
\frac{{\partial {z_p}(\omega ,t)}}{{\partial t}} =& \hspace{-8pt} - \omega {z_p}(\omega ,t) - {{\mathscr D}^\gamma }\hat p,\\
\tilde p =& \hspace{-8pt} \int_0^\infty  {{\mu _\gamma }(\omega ){z_p}(\omega ,t){\rm d}\omega },
\end{array} \right.
\end{eqnarray}
with ${\mu _\gamma }{\rm{(}}\omega {\rm{) = }}\frac{{\sin (\gamma \pi )}}{{{\omega ^\gamma }\pi }}$.

With the definition of
\begin{eqnarray}\label{Eq3-4-9}
\left\{ \begin{array}{l}
w = \hat p\bar w,\\
v = \hat p\bar v,
\end{array} \right.
\end{eqnarray}
then we have
\begin{eqnarray}\label{Eq3-4-10}
\begin{array}{r}
\bar bw - \frac{1}{{\hat p}}\Delta w = \bar w - \bar b\tilde p\bar w - (\bar w - \bar v) = \bar v - \bar b\tilde p\bar w.
\end{array}
\end{eqnarray}

The derivative of $\varepsilon_n$ is
\begin{eqnarray}\label{Eq3-4-11}
\begin{array}{rl}
{{\mathscr D}^{{\alpha _n}}}{\varepsilon _n} =& \hspace{-8pt} {{\mathscr D}^{{\alpha _n}}}{{\bar x}_n} - {{\mathscr D}^{\sum\limits_{j = 1}^n {{\alpha _j}} }}r - {{\mathscr D}^{{\alpha _n}}}{\tau _{n - 1}} - {{\mathscr D}^{{\alpha _n}}}{\lambda _n}\\
 =& \hspace{-8pt} \bar bw + {{\bar \psi }_n} + \bar \varphi _n^{\rm{T}}\theta  + \bar d - {{\mathscr D}^{\sum\limits_{j = 1}^n {{\alpha _j}} }}r \\
& \hspace{-8pt} - {{\mathscr D}^{{\alpha _n}}}{\tau _{n - 1}} - \frac{1}{{\hat p}}\Delta w + {c_n}{\lambda _n}\\
 =& \hspace{-8pt} \bar v - \bar b\tilde p\bar w + {{\bar \psi }_n} + \bar \varphi _n^{\rm{T}}\theta  + \bar d - {{\mathscr D}^{\sum\limits_{j = 1}^n {{\alpha _j}} }}r \\
& \hspace{-8pt} - {{\mathscr D}^{{\alpha _n}}}{\tau _{n - 1}} +  {c_n}{\lambda _n}.
\end{array}
\end{eqnarray}

It's corresponding frequency distributed model is
\begin{eqnarray}\label{Eq3-4-12}
\left\{ \begin{array}{rl}
\frac{{\partial {z_n}(\omega ,t)}}{{\partial t}} =& \hspace{-8pt} - \omega {z_n}(\omega ,t) + \bar v - \bar b\tilde p\bar w + {{\bar \psi }_n} + \bar \varphi _n^{\rm{T}}\theta  \\
& \hspace{-8pt} + \bar d - {{\mathscr D}^{\sum\limits_{j = 1}^n {{\alpha _j}} }}r - {{\mathscr D}^{{\alpha _n}}}{\tau _{n - 1}}  + {c_n}{\lambda _n},\\
{\varepsilon _n} = & \hspace{-8pt} \int_0^\infty  {{\mu _{{\alpha _n}}}(\omega ){z_n}(\omega ,t){\rm{d}}\omega },
\end{array} \right.
\end{eqnarray}
with ${\mu _{{\alpha _n}}}{\rm{(}}\omega {\rm{) = }}\frac{{\sin ({\alpha _n}\pi )}}{{{\omega ^{{\alpha _n}}}\pi }}$.

Selecting the Lyapunov function
\begin{eqnarray}\label{Eq3-4-13}
\begin{array}{rl}
{V_n} =& \hspace{-8pt} {V_{n - 1}} + \frac{1}{2}\int_0^\infty  {{\mu _{{\alpha _n}}}(\omega )z_n^2(\omega ,t){\rm{d}}\omega }  \\
& \hspace{-8pt} + \frac{{\left| {\bar b} \right|}}{{2\eta }}\int_0^\infty  {{\mu _\gamma }(\omega )z_p^2(\omega ,t){\rm{d}}\omega } \\
& \hspace{-8pt} + \frac{1}{{2\xi }}\int_0^\infty  {{\mu _\rho }(\omega )z_D^2(\omega ,t){\rm{d}}\omega },
\end{array}
\end{eqnarray}
and adopting the preceding  derivative $\dot V_{n-1}$, the derivative of $ V_{n}$ is
\begin{eqnarray}\label{Eq3-4-14}
\begin{array}{rl}
{{\dot V}_n} =& \hspace{-8pt} {{\dot V}_{n - 1}} - \int_0^\infty  {\omega {\mu _{{\alpha _n}}}(\omega )z_n^2(\omega ,t){\rm{d}}\omega }  \\
& \hspace{-8pt} - \frac{{\left| {\bar b} \right|}}{\eta }\int_0^\infty  {\omega {\mu _\gamma }(\omega )z_p^2(\omega ,t){\rm{d}}\omega } \\
& \hspace{-8pt} - \frac{1}{\xi }\int_0^\infty  {\omega {\mu _\rho }(\omega )z_D^2(\omega ,t){\rm{d}}\omega } + {\varepsilon _n}( - \bar b\tilde p\bar w + {{\bar \psi }_n} \\
& \hspace{-8pt} + \bar \varphi _n^{\rm{T}}\theta  + \bar d  - {{\mathscr D}^{\sum\limits_{j = 1}^n {{\alpha _j}} }}r - {{\mathscr D}^{{\alpha _n}}}{\tau _{n - 1}} + \bar v + {c_n}{\lambda _n})\\
& \hspace{-8pt} - \frac{{\left| {\bar b} \right|}}{\eta }\tilde p{{\mathscr D}^\gamma }\hat p - \frac{1}{\xi }\tilde D{{\mathscr D}^\rho }\hat D\\
 \le & \hspace{-8pt} - \sum\limits_{j = 1}^n {\int_0^\infty  {\omega {\mu _{{\alpha _j}}}(\omega )z_j^2(\omega ,t){\rm{d}}\omega } }  \\
& \hspace{-8pt} - \int_0^\infty  {\omega {\mu _\beta }(\omega )z_\theta ^{\rm{T}}(\omega ,t){\Lambda ^{ - 1}}{z_\theta }(\omega ,t){\rm{d}}\omega }  \\
& \hspace{-8pt} - \frac{{\left| {\bar b} \right|}}{\eta }\int_0^\infty  {\omega {\mu _\gamma }(\omega )z_p^2(\omega ,t){\rm{d}}\omega } \\
& \hspace{-8pt} - \frac{1}{\xi }\int_0^\infty  {\omega {\mu _\rho }(\omega )z_D^2(\omega ,t){\rm{d}}\omega } \\
& \hspace{-8pt} - \sum\limits_{j = 1}^{n - 1} {{{\bar c}_j}\varepsilon _j^2} + \frac{1}{2}\varepsilon _n^2 - \tilde p\left| {\bar b} \right|[\frac{1}{\eta }{{\mathscr D}^\gamma }\hat p + {\varepsilon _n}{\mathop{\rm sgn}} (\bar b)\bar w] \\
& \hspace{-8pt} + {{\tilde \theta }^{\rm{T}}}(\sum\limits_{j = 1}^n {{\varepsilon _j}{{\bar \varphi }_j}}  - {\Lambda ^{ - 1}}{{\mathscr D}^\beta }\hat \theta )   + {\varepsilon _n}[{{\bar \psi }_n} + \bar \varphi _n^{\rm{T}}\hat \theta \\
& \hspace{-8pt} + {\mathop{\rm sgn}} ({\varepsilon _n})\hat D - {{\mathscr D}^{\sum\limits_{j = 1}^n {{\alpha _j}} }}r  - {{\mathscr D}^{{\alpha _n}}}{\tau _{n - 1}} + \bar v  \\
& \hspace{-8pt}  + {c_n}{\lambda _n}] + {\varepsilon _n}\bar d - \left| {{\varepsilon _n}} \right|\hat D - \frac{1}{\xi }\tilde D{{\mathscr D}^\rho }\hat D.
\end{array}
\end{eqnarray}

Taking the parameter update law (\ref{Eq3-4-4}) and adaptive control law (\ref{Eq3-4-5}) into account, the above derivative $\dot V_n$ will be simplified
\begin{eqnarray}\label{Eq3-4-15}
\begin{array}{rl}
{{\dot V}_n} \le & \hspace{-8pt} - \sum\limits_{j = 1}^n {\int_0^\infty  {\omega {\mu _{{\alpha _j}}}(\omega )z_j^2(\omega ,t){\rm{d}}\omega } }  \\
& \hspace{-8pt} - \int_0^\infty  {\omega {\mu _\beta }(\omega )z_\theta ^{\rm{T}}(\omega ,t){\Lambda ^{ - 1}}{z_\theta }(\omega ,t){\rm{d}}\omega } \\
& \hspace{-8pt} - \frac{{\left| {\bar b} \right|}}{\eta }\int_0^\infty  {\omega {\mu _\gamma }(\omega )z_p^2(\omega ,t){\rm{d}}\omega } \\
& \hspace{-8pt} - \frac{1}{\xi }\int_0^\infty  {\omega {\mu _\rho }(\omega )z_D^2(\omega ,t){\rm{d}}\omega }  - \sum\limits_{j = 1}^n {{{\bar c}_j}\varepsilon _j^2},
\end{array}
\end{eqnarray}
where ${{\bar c}_n} = {c_n} - \frac{1}{2} > 0$.

Based on the LaSalle invariant principle, the error variable $\varepsilon_i$ are convergent to zero by the similar proof of Theorem 1, which establishes Theorem 2.
\end{proof}

Just like Theorem 1, in the recursive procedure of controllers design, every coefficient $c_i$ is required to be greater than a certain constant with the purpose of establishing inequality and eliminating the product ${\varepsilon_i}{\varepsilon_{i+1}}$. This procedure also leads to the conservatism of application, though relatively superior control performance will be realized afterwards. To meet the different needs of control, a more general and flexible FOABC strategy is derived from Theorem 2 and Corollary 1.


\begin{corollary}\label{Corollary 2}
For the purpose of compensation, a fractional order auxiliary system is designed to generate virtual signals $\lambda  = {[~{\lambda _1}, {\lambda _2}, \cdots , {\lambda _n}~]^{\rm T}}$
\begin{eqnarray}\label{Eq3-4-16}
\left\{ \begin{array}{rl}
{{\mathscr D}^{{\alpha _i}}}{\lambda _i} =& \hspace{-8pt} {\lambda _{i + 1}} - {a_i}{\mathop{\rm sgn}} ({\lambda _i}){\left| {{\lambda _i}} \right|^{{\mu _i}}},\\
& \hspace{-8pt} i = 1, \cdots ,n - 1,\\
{{\mathscr D}^{{\alpha _n}}}{\lambda _n} =& \hspace{-8pt} \frac{1}{{\hat p}}\Delta w - {a_n}{\mathop{\rm sgn}} ({\lambda _n}){\left| {{\lambda _n}} \right|^{{\mu _n}}}.
\end{array} \right.
\end{eqnarray}
There is a control method that consists of

\noindent the error variables
\begin{eqnarray}\label{Eq3-4-17}
\left\{ \begin{array}{l}
{\varepsilon _1} = {{\bar x}_1} - r - {\lambda _1},\\
{\varepsilon _i} = {{\bar x}_i} - {{\mathscr D}^{\sum\limits_{j = 1}^{i - 1} {{\alpha _j}} }}r - {\tau _{i - 1}} - {\lambda _i}, ~i = 2, \cdots ,n,
\end{array} \right.
\end{eqnarray}
the stabilizing functions
\begin{eqnarray}\label{Eq3-4-18}
\left\{ \begin{array}{rl}
{\tau _1} =& \hspace{-8pt} - {c_1}{\mathop{\rm sgn}} ({\varepsilon _1}){\left| {{\varepsilon _1}} \right|^{{\sigma _1}}} - {{\bar \psi }_1} - \bar \varphi _1^{\rm T}\hat \theta  \\
& \hspace{-8pt} - {a_1}{\mathop{\rm sgn}} ({\lambda _1}){\left| {{\lambda _1}} \right|^{{\mu _1}}},\\
{\tau _i} =& \hspace{-8pt} - {\varepsilon _{i - 1}} - {c_i}{\mathop{\rm sgn}} ({\varepsilon _i}){\left| {{\varepsilon _i}} \right|^{{\sigma _i}}} - {{\bar \psi }_i} - \bar \varphi _i^{\rm T}\hat \theta   \\
& \hspace{-8pt} + {{\mathscr D}^{{\alpha _i}}}{\tau _{i - 1}} - {a_i}{\mathop{\rm sgn}} ({\lambda _i}){\left| {{\lambda _i}} \right|^{{\mu _i}}},\\
& \hspace{-8pt} i = 2, \cdots ,n - 1,
\end{array} \right.
\end{eqnarray}
the parameter update law
\begin{eqnarray}\label{Eq3-4-19}
\left\{ \begin{array}{l}
{{\mathscr D}^\beta }\hat \theta  = \Lambda \sum\limits_{j = 1}^n {{\varepsilon _j}{{\bar \varphi }_j}}, \\
{{\mathscr D}^\gamma }\hat p =  - \eta {\mathop{\rm sgn}} (\bar b){\varepsilon _n}\bar w,\\
{{\mathscr D}^\rho }\hat D = \xi \left| {{\varepsilon _n}} \right|,
\end{array} \right.
\end{eqnarray}
the adaptive control law
\begin{eqnarray}\label{Eq3-4-20}
\left\{ \begin{array}{rl}
v =& \hspace{-8pt} \hat p\bar v,\\
\bar v =& \hspace{-8pt} - {\varepsilon _{n - 1}} - {c_n}{\mathop{\rm sgn}} ({\varepsilon _n}){\left| {{\varepsilon _n}} \right|^{{\sigma _n}}} - {\mathop{\rm sgn}} ({\varepsilon _n})\hat D \\
& \hspace{-8pt} + {{\mathscr D}^{\sum\limits_{j = 1}^n {{\alpha _j}} }}r + {{\mathscr D}^{{\alpha _n}}}{\tau _{n - 1}} - {{\bar \psi }_n} - \bar \varphi _n^{\rm{T}}\hat \theta \\
& \hspace{-8pt} - {a_n}{\mathop{\rm sgn}} ({\lambda _n}){\left| {{\lambda _n}} \right|^{{\mu _n}}}.
\end{array} \right.
\end{eqnarray}
Then all the signals in the closed-loop adaptive system are globally uniformly bounded, and the asymptotic tracking is achieved as
\begin{eqnarray}\label{Eq3-4-21}
\mathop {\lim }\limits_{t \to \infty } [y(t) - r(t)] = 0,
\end{eqnarray}
where $\hat p$ is the parameter estimate of $\frac{1}{{\bar b}}$, $\hat \theta$ is the parameter estimate of $\theta$, $\hat D$ is the parameter estimate of $\bar D$, $\Lambda  \in {{\mathbb R}^{q \times q}}$ is a positive definite matrix, $0< \gamma , \beta , \rho , \alpha _i , \sigma_i , \mu_i<1$ , $\eta , \xi ,  a_i , c_i>0 ~(i=1, \cdots, n)$, $w = \hat p\bar w$ and $\Delta w = w - v$.
\end{corollary}

The proof of Corollary 2 is omitted as well, since one can easily complete it according to the procedure of Theorem \ref{Theorem 2}.

The proposed theorems and corollaries provide a new and original approach that the control input subject to saturation and dead zone could be coped with separately. By applying adaptive backstepping control strategy, the controller design is completed in the procedure of proof afterwards. To highlight the characteristics and advantages of the proposed control strategy clearly, the following remark is presented here, which is also part of our contributions.


\begin{remark}\label{Remark 4}
There are four additional innovations needing emphasis.
\begin{itemize}
  \item With the introduction of intermediate variable $w$ and projection, the control input subject to nonsmooth nonlinear dead zone and saturation is decomposed into two independent parts. The desired input $v$ goes through saturation part and changes into intermediate variable $w$, then $w$ changes into the actual input $u$ after dead zone part. As the bounded disturbance $d'$ combines with a bounded term $d''$ of dead zone part, the disturbance-like term forms, and finally, the problem of dead zone and saturation is transformed into that of disturbance and saturation. It is worth mentioning that when the $\bar b$ is unknown, the slope of input $m$ may be also unknown, that is, all parameters of input nonsmooth nonlinearities are unknown, but our proposed method still works according to Theorem \ref{Theorem 2} and Corollary \ref{Corollary 2}. As far as we known, there isn't any similar solution to dead zone and saturation no matter whether the object is fractional order or not.
  \item In Corollary \ref{Corollary 1} and Corollary \ref{Corollary 2}, instead of establishing the inequality just as Theorem \ref{Theorem 1} and Theorem \ref{Theorem 2}, the product ${\varepsilon_i}{\varepsilon_{i+1}}$ is eliminated directly by adding the error variables ${\varepsilon_i}$ into the stabilizing functions $\tau _i$ and the adaptive control law $v$, which reduces the restraint of coefficients $c_i$. Compared with the traditional linear feedback elements ${c_i}{\varepsilon _i}$ and ${c_i}{\lambda _i}$, the nonlinear elements ${c_i}{\mathop{\rm sgn}} ({\varepsilon _i}){\left| {{\varepsilon _i}} \right|^{{\sigma _i}}}$ and ${a_i}{\mathop{\rm sgn}} ({\lambda _i}){\left| {{\lambda _i}} \right|^{{\mu _i}}}$ bring some advantages, namely larger error corresponds to smaller gain, but smaller error gives larger gain, which requires smaller control cost and reduces  overshoot caused by sudden change of reference signal. In addition, controllers design will enjoy more degree of freedom since new parameters are introduced.
  \item The theorems and corollaries realize tracking and compensate the nonlinearities of incommensurate FOS. If $\alpha_i=\alpha$ these methods change into the commensurate case. Furthermore, if $\alpha=1$, the above theories will be common solutions to integer order systems with nonsmooth nonlinearities. When reference signal equals zero $r=0$, the tracking task turns into a stabilizing task.
  \item Instead of fractional order derivative of the Lyapunov function \cite{Ding:2014CTA,Ding:2015NODY}, indirect Lyapunov method with frequency distributed model is adopted so that the order of the parameter update law is no longer fixed to the system order. Its design enjoys more degree of freedom and it is expected to achieve better control performance. By setting the order $\beta$ and $\gamma$ equal to the integer, the update law changes back.
\end{itemize}
\end{remark}

\section{Simulation Study}\label{Section 4}
Numerical simulations will be carried out in this section with the piecewise numerical approximation algorithm \cite{Wei:2014IJCAS}.
The specific form of controlled plant is chosen as
\begin{eqnarray}\label{Eq4-1}
\left\{ \begin{array}{l}
{{\mathscr D}^{0.5}}{x_1} = 2{x_2} - 0.5x_1^2 + {x_1}\theta ,\\
{{\mathscr D}^{0.6}}{x_2} = -{x_3} + \frac{{{x_2} - x_2^3}}{{1 + x_1^4}} - \sin \left( {{x_1}} \right)\theta ,\\
{{\mathscr D}^{0.7}}{x_3} = bu(v) - {{\rm{e}}^{ - x_1^2}}\sin \left( {5{x_3}} \right) + \cos \left( {{x_1}} \right)\theta + d,
\end{array} \right.
\end{eqnarray}
with the unknown constant $\theta=0.1$, bounded disturbance $d={\rm cos}(\pi t)+{\rm sin}(3t)$, possibly unknown control input coefficient $b=3$, the unknown parameters of nonsmooth nonlinearities $U_{up}=1.8$, $U_{low}=-1.5$, $b_{r}=0.8$, $b_{l}=-0.5$ and possibly unknown slope $m=1$. A sinusoidal signal $r(t)={\rm{sin}}(2t)$ is set to reference signal and will be tracked by the output of the system with initial state $x(0)=\left[~0.3,~-0.4,~~0.2~\right]^{\rm {T}}$. The FOTD could be configured as required, here $r_1=50, r_2=5$, by which the fractional order derivative of stabilizing function $\tau_i$ and reference signal $r$ , like ${\mathscr D}^{\alpha_i} \tau_{i-1}$ and ${{\mathscr D}^{\sum\nolimits_{j = 1}^{i - 1} {{\alpha _j}} }}r$, will be obtained easily. The $\rm{tanh}(\cdot)$ takes place of $\rm{sgn}(\cdot)$ in control law $v$ for smooth calculation and chattering rejection.

To demonstrate the applicability of FOABC method, the system (\ref{Eq4-1}) with or without the certainty of $b$ and $m$ will be both investigated, and results are presented in Example 1 and Example 2, respectively.

\vspace{10pt}
\noindent Example 1. Supposing the coefficient $\bar b$ of the system (\ref{Eq4-1}) is known, the case 1 adopts the method of Theorem \ref{Theorem 1}, the case 2 adopts the method of Corollary \ref{Corollary 1}, and the case 3 is only carried out with the general FOABC method \cite{Wei:2015NEU} where no solution is provided for nonsmooth nonlinearities. Selecting the controller parameters $c_1=c_2=c_3=a_1=a_2=a_3=4$, $\sigma_1=\sigma_2=\sigma_3=\mu_1=\mu_2=\mu_3=0.8$, $\rho = \beta=0.9$, $\xi = \Lambda=1$ and the initial parameter estimates $\hat D(0) = \hat \theta(0)=0$, then we view $\varepsilon(t) = r(t) - y(t)$ as tracking error and obtain the tracking performance shown in Fig. \ref{Fig 4}. Meanwhile, the Fig. \ref{Fig 5} and Fig. \ref{Fig 6} present the corresponding estimations $\hat \theta$ and $\hat D$, respectively. The actual control input is drawn in Fig. \ref{Fig 7}.
\begin{figure}[!htbp]
  \centering
  \includegraphics[width=0.8\hsize]{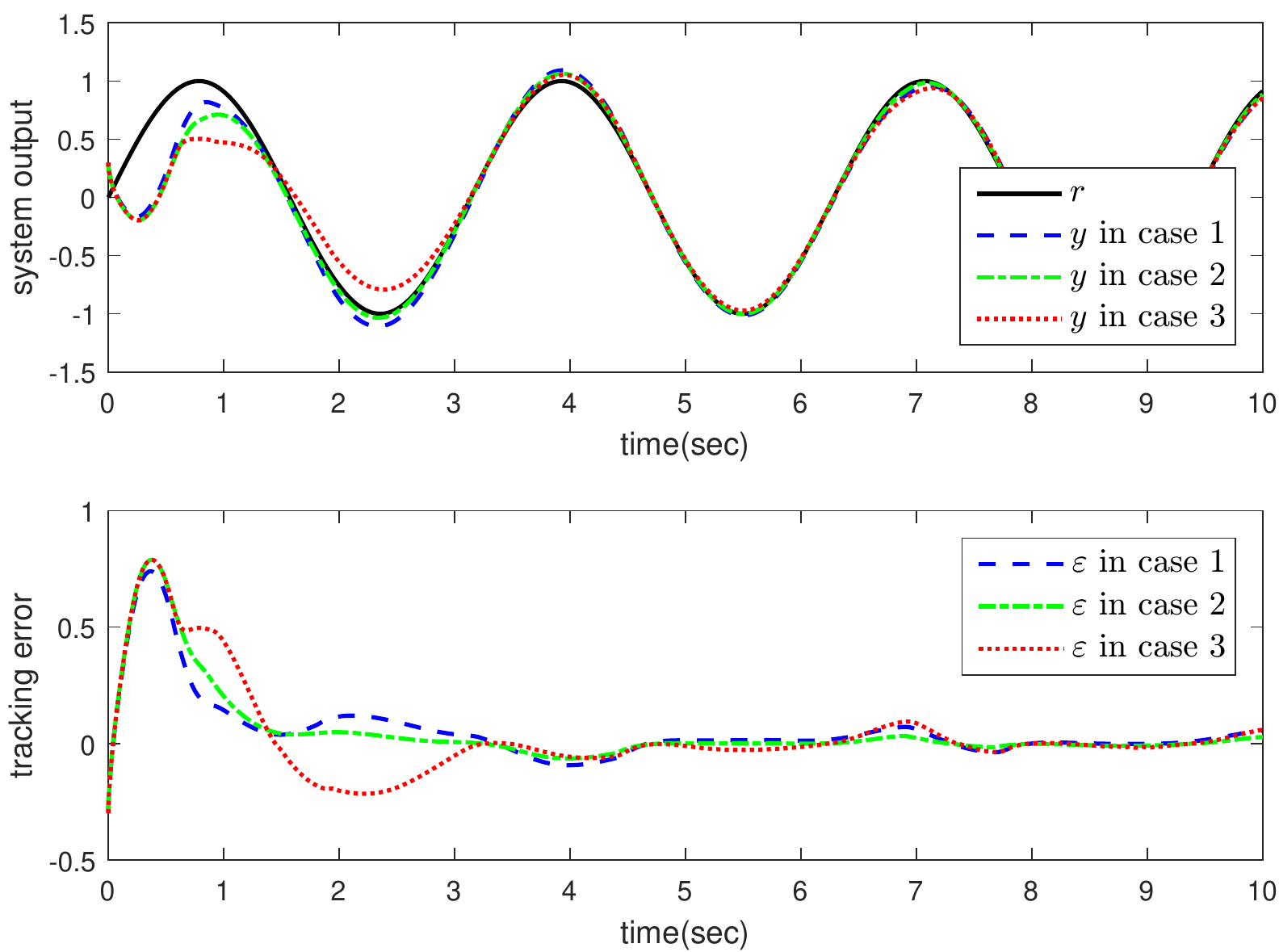}
  \caption{The performance of tracking in Example 1.}\label{Fig 4}
\end{figure}
\begin{figure}[!htbp]
  \centering
  \includegraphics[width=0.8\hsize]{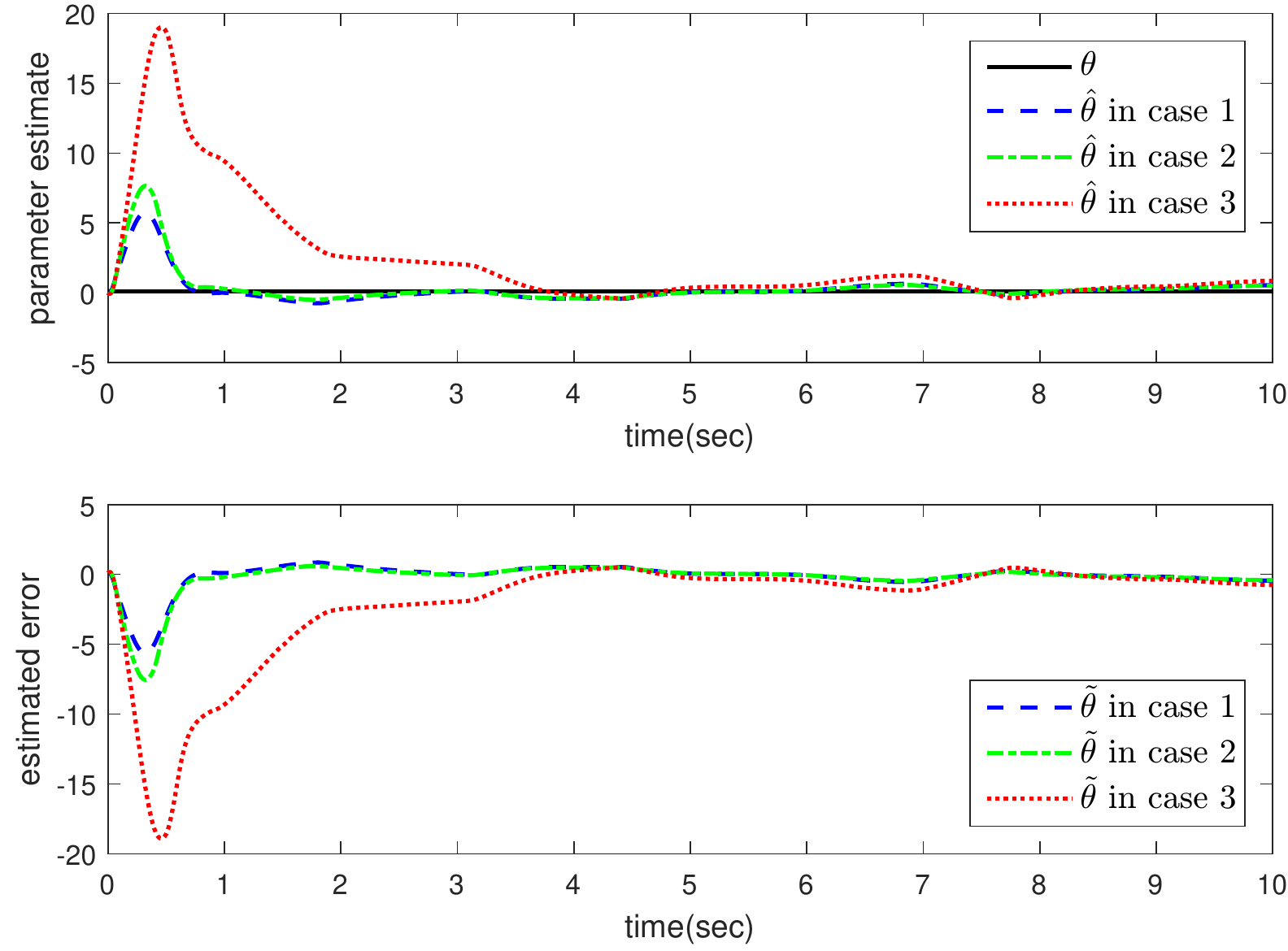}
  \caption{The estimation of $\theta$ in Example 1.}\label{Fig 5}
\end{figure}
\begin{figure}[!htbp]
  \centering
  \includegraphics[width=0.8\hsize]{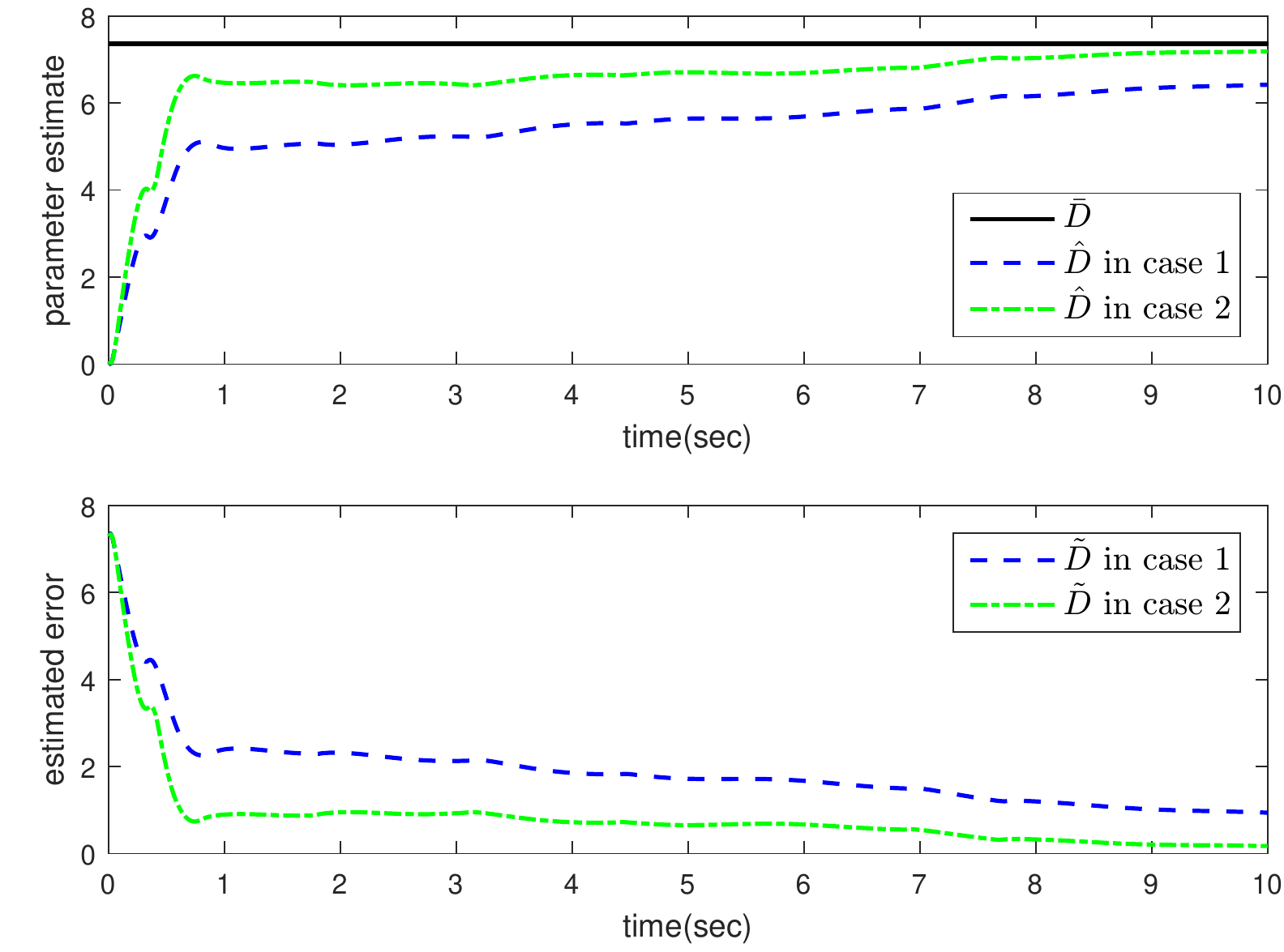}
  \caption{The relevant estimation of $\bar D$ in Example 1.}\label{Fig 6}
\end{figure}
\begin{figure}[!htbp]
  \centering
  \includegraphics[width=0.8\hsize]{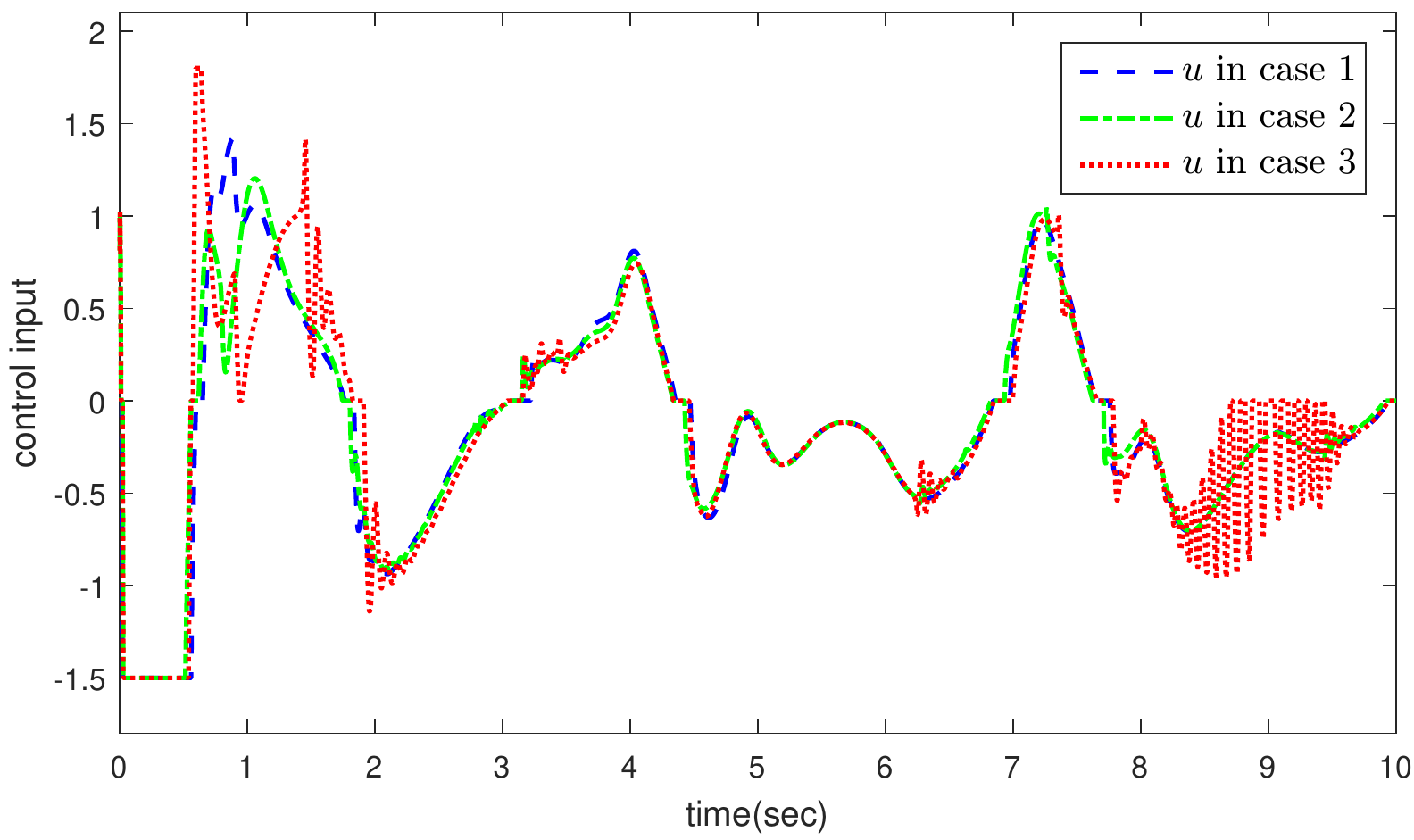}
  \caption{The needed control input in Example 1.}\label{Fig 7}
\end{figure}

The output $y(t)$, in view of the Fig. \ref{Fig 4}, could track the reference signal $r(t)$ for all three cases. However, compared with the case 2 and case 3, the case 1 shows better tracking result, because the inequality and restraint of coefficient make the Lyapunov function stricter and more conservative. Besides better tracking performance, the parameter estimation also shows the superiority. Although there is a larger tracking error at the beginning, the output of case 2 also converges to reference signal more quickly than that of case 3. And case 2 shows similar superiority of parameter estimation, too. Compared with the method of case 1, the controllers design of case 2 enjoys more degree of freedom, which may meet the different needs of practical project. For quantitative comparison, some details are calculated as Table \ref{Tab 1}.
\begin{table}[!ht]
\renewcommand\arraystretch{1.3}
\centering
\caption{The control performance of Example 1}\label{Tab 1}
\begin{tabular}{cccccc}
\hline
\hline
 & $|| {\varepsilon (t)} ||_{\rm {max}}$ & $|| {\varepsilon (t)} ||_2$ & $||\tilde \theta (t)||_2$ & $||\tilde D (t)||_2$ & $|| {u(t)} ||_2$ \\
\hline
case 1 & 0.73965 & 18.084 & 114.58 & 245.41 & 67.604 \\
case 2 & 0.78887 & 19.639 & 145.80 & 154.38 & 64.233 \\
case 3 & 0.78892 & 23.036 & 514.81 & $\backslash$ & 67.396 \\
\hline
\hline
\end{tabular}
\end{table}

Since the disturbance hasn't been considered in the controlled plant of general FOABC method \cite{Wei:2015NEU}, it is short of ability to estimate $\bar D$, which results in no data in Fig. \ref{Fig 6} and Table \ref{Tab 1}.

\vspace{0.3cm}
\noindent Example 2. Supposing the coefficient $\bar b$ of the system (\ref{Eq4-1}) is unknown, the Example 2 will be executed by the same parameters as Example 1 except $\gamma=0.9$, $\hat p(0)=0.01$ and $\eta  = 1$. Because $\hat p$ will be regarded as denominator during the calculation, the initial value $\hat p(0)$ cannot be equal to zero. The case 1 adopts the method of Theorem \ref{Theorem 2}. The case 2 adopts the method of Corollary \ref{Corollary 2}. The case 3 is only carried out with the general FOABC method \cite{Wei:2015NEU} which will be compared with our method shown in case 1 and case 2. The curves of output and tracking error are drawn on Fig. \ref{Fig 8}. As for estimation, Fig. \ref{Fig 9}, Fig. \ref{Fig 10} and  Fig. \ref{Fig 11} express the the parameter estimate of $\theta$, $p$ and $\bar D$, respectively. The synergetic control input is shown in Fig. \ref{Fig 12}.
\begin{figure}[!htbp]
  \centering
  \includegraphics[width=0.8\hsize]{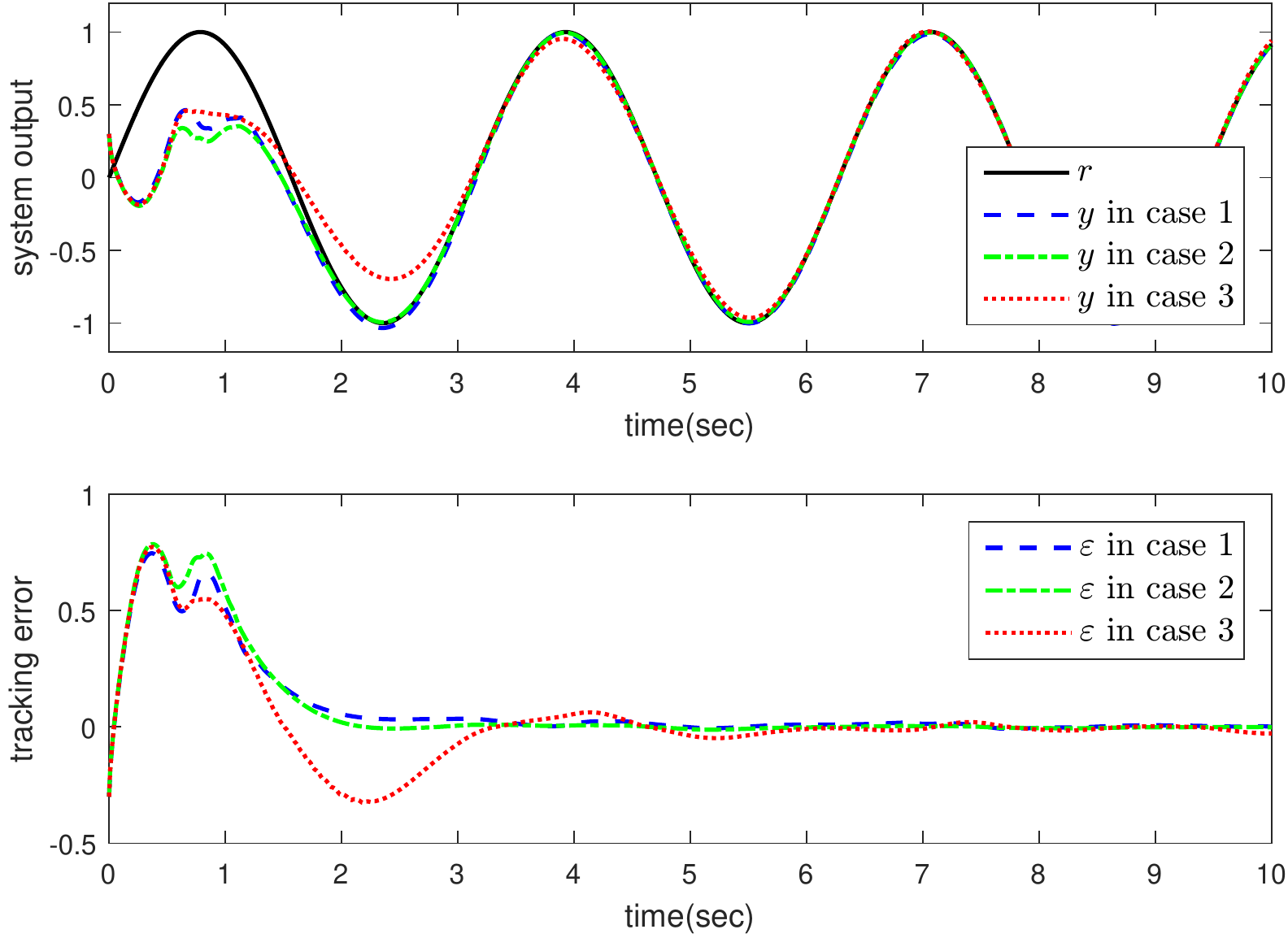}
  \caption{The performance of tracking in Example 2.}\label{Fig 8}
\end{figure}
\begin{figure}[!htbp]
  \centering
  \includegraphics[width=0.8\hsize]{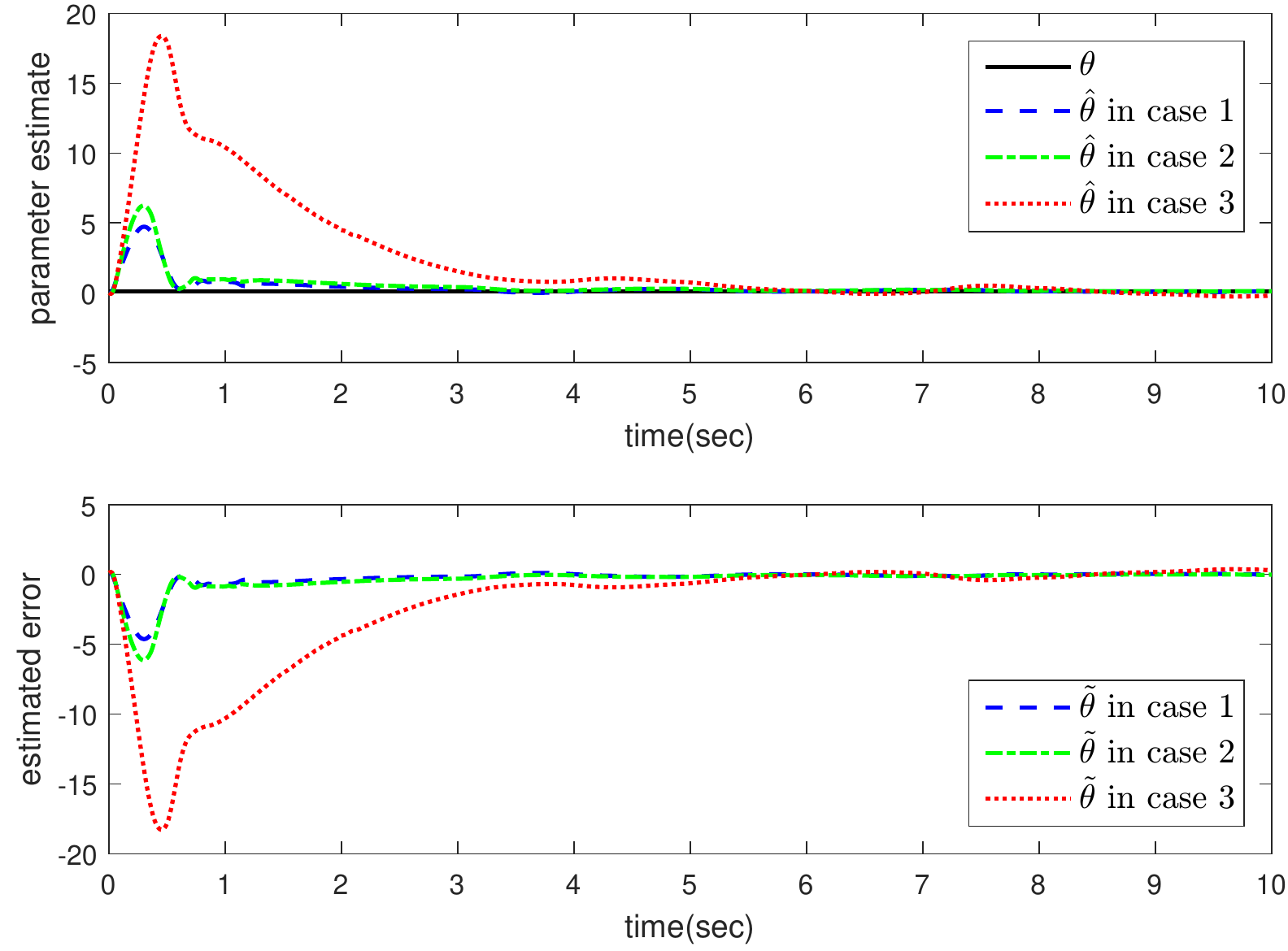}
  \caption{The estimation of $\theta$ in Example 2.}\label{Fig 9}
\end{figure}
\begin{figure}[!htbp]
  \centering
  \includegraphics[width=0.8\hsize]{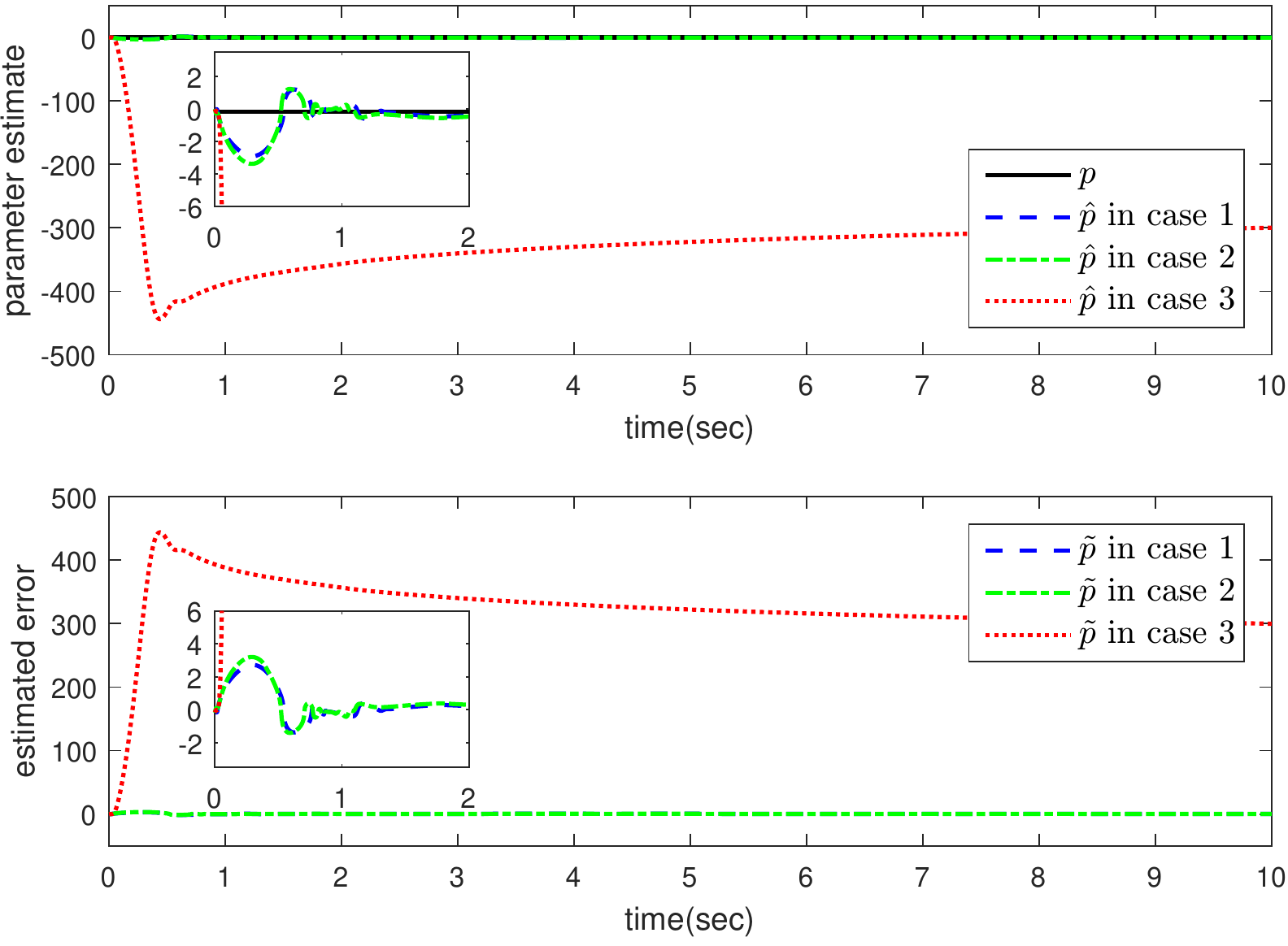}
  \caption{The relevant estimation of $p$ in Example 2.}\label{Fig 10}
\end{figure}
\begin{figure}[!htbp]
  \centering
  \includegraphics[width=0.8\hsize]{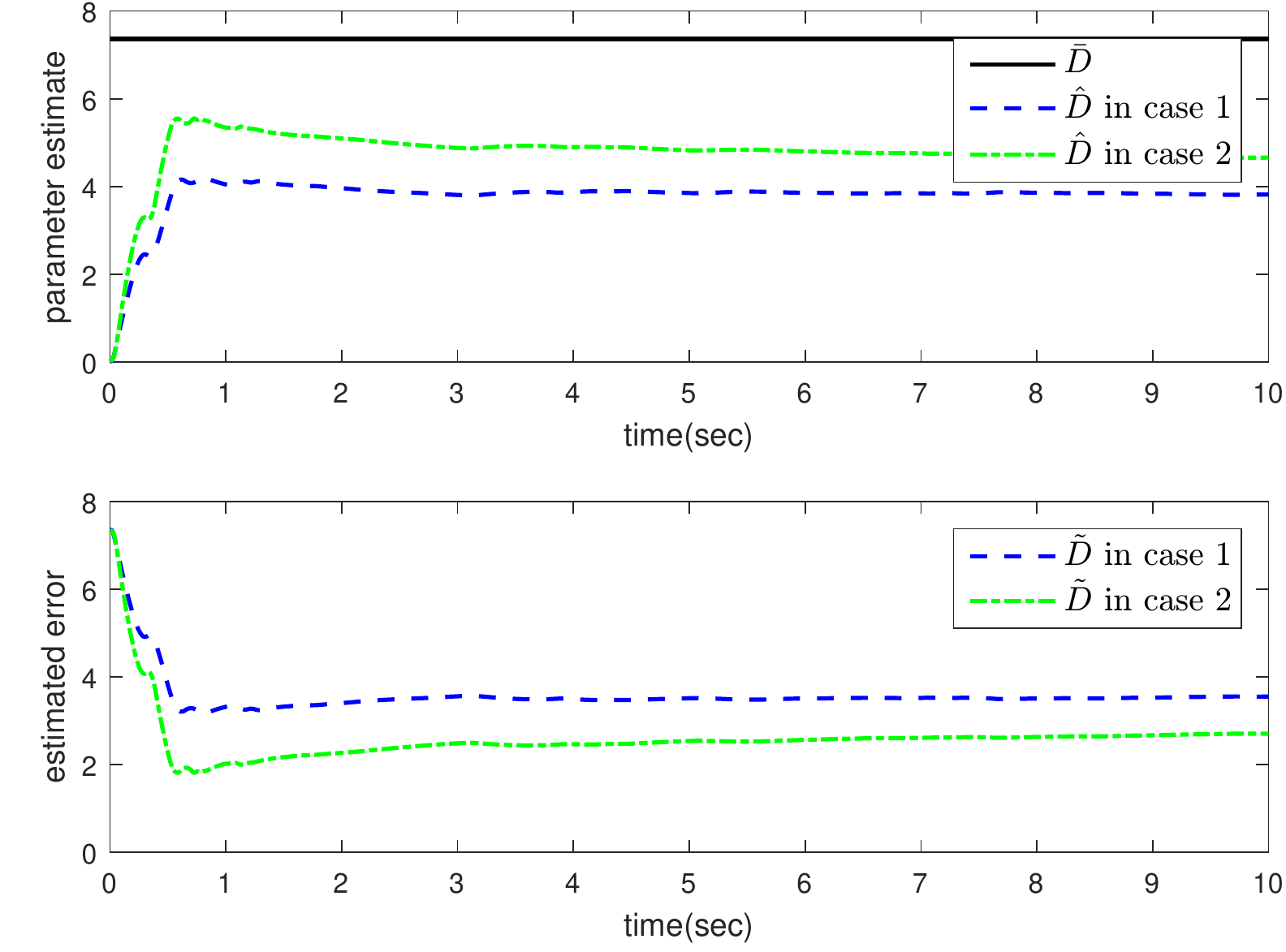}
  \caption{The relevant estimation of $\bar D$ in Example 2.}\label{Fig 11}
\end{figure}
\begin{figure}[!htbp]
  \centering
  \includegraphics[width=0.8\hsize]{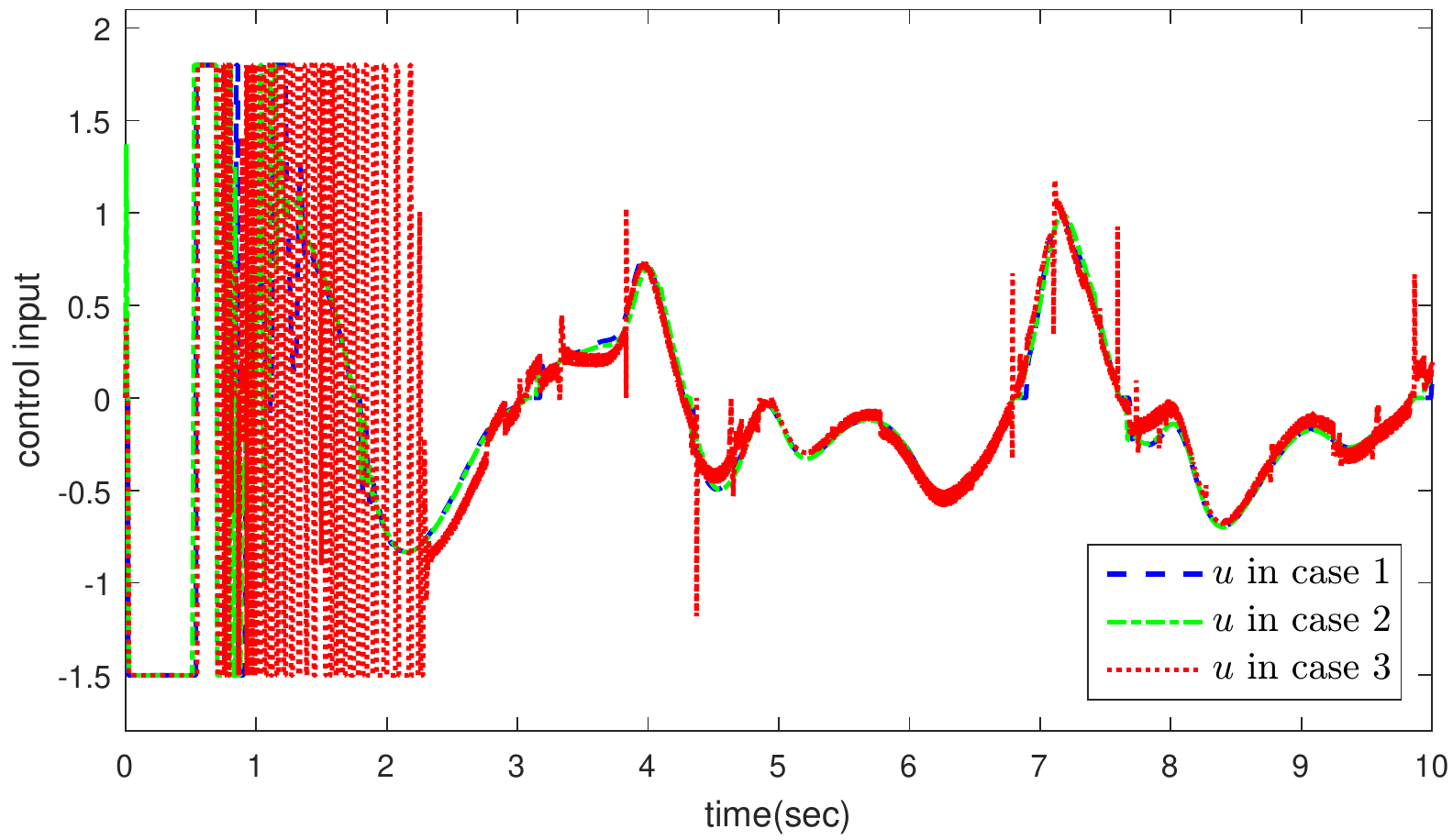}
  \caption{The needed control input in Example 2.}\label{Fig 12}
\end{figure}

The tracking performance in this example, analogous to Example 1, exposes the high speed of our methods that output could keep pace with the reference signal rapidly. What's more, our estimating procedure illustrates that the methods we suggest have excellent advantage. The case 1 still shows its superiority for the same reason as Example 1. The controllers of case 2 also enjoys more freedom. Some details are also calculated as Table \ref{Tab 2}.
\begin{table*}[hbt]
\renewcommand\arraystretch{1.3}
\centering
\caption{The control performance of Example 2}\label{Tab 2}
\begin{tabular}{ccccccc}
\hline
\hline
  & $|| {\varepsilon (t)} ||_{\rm {max}}$ & $|| {\varepsilon (t)} ||_2$ & $|| {\tilde \theta (t)} ||_2$ & $|| {\tilde p (t)} ||_2$ & $|| {\tilde D (t)} ||_2$ & $|| {u(t)} ||_2$ \\
\hline
case 1 & 0.74547 & 39.939 & 116.67 & 209.591 & 2193.7 & 227.30 \\
case 2 & 0.78489 & 43.165 & 161.70 & 205.557 & 1599.4 & 225.35 \\
case 3 & 0.77238 & 55.456 & 1007.3 & 199644 & $\backslash$ & 230.88 \\
\hline
\hline
\end{tabular}
\end{table*}

As can be observed from Table \ref{Tab 1} and Table \ref{Tab 2}, compared with the system with known coefficient, the unknown case needs more control cost in the same circumstances of simulation, but only gets barely satisfactory tracking performance. Compared with the case 1, the case 2 requires less control cost because of the nonlinear feedback element.

The Example 2 proceeds under the condition of unknown $\bar b$. Because of $\bar b = \delta_n b m$, the $b$ or $m$ is uncertain, too. That is to say, our proposed method still works well without prior knowledge of all parameters of nonsmooth nonlinearities, namely unknown $m$, $U_{up}$, $U_{low}$, $b_r$ and $b_l$. Therefore, it leads to higher versatility and wider application.

\section{Conclusions}\label{Section 5}

The input saturation and dead zone of uncertain nonlinear FOS are investigated in this paper. After the decomposition of control input, the problem of input saturation and dead zone is transformed into the problem of input saturation and bounded disturbance, which could be solved in accordance with FOABC. It is the first time that scholars have studied the FOS with input saturation and dead zone at the same time. What's more, our proposed method could still work even if all parameters of these nonsmooth nonlinearities are unknown. It is believed that this novel method provides a new way to cope with control input subject to saturation and dead zone.

\section*{Acknowledgements}
\phantomsection
\addcontentsline{toc}{section}{Acknowledgements}
The work described in this paper was fully supported by the National Natural Science Foundation of China (No. 61573332, No. 61601431), the Fundamental Research Funds for the Central Universities (No. WK2100100028), the Anhui Provincial Natural Science Foundation (No. 1708085QF141) and the General Financial Grant from the China Postdoctoral Science Foundation (No. 2016M602032).

\section*{References}
\phantomsection
\addcontentsline{toc}{section}{References}
\bibliographystyle{model3-num-names}
\bibliography{database}







\end{document}